\documentclass[a4paper,10pt,american]{amsart}
\usepackage{amssymb}
\usepackage{dsfont}
\usepackage{stmaryrd}
\usepackage{esint}

\usepackage{xcolor}
\usepackage[colorlinks,pdfpagelabels,pdfstartview = FitH,bookmarksopen
= true,bookmarksnumbered = true,linkcolor = blue,plainpages =
false,hypertexnames = false,citecolor = red,pagebackref=false]{hyperref}
\usepackage{enumitem}
\usepackage{changes}

\allowdisplaybreaks

\newtheorem{theorem}{Theorem}
\newtheorem{lemma}[theorem]{Lemma}
\newtheorem{proposition}[theorem]{Proposition}
\newtheorem{corollary}[theorem]{Corollary}

\theoremstyle{definition}
\newtheorem{definition}[theorem]{Definition}

\theoremstyle{remark}
\newtheorem{remark}[theorem]{Remark}

\numberwithin{theorem}{section}
\numberwithin{equation}{section}

\newcommand{\N}{\mathds{N}}
\newcommand{\R}{\mathds{R}}

\newcommand{\F}{\mathcal{F}}

\renewcommand{\L}{\mathcal{L}}

\renewcommand{\b}{\mathfrak{b}}

\renewcommand{\d}{\mathrm{d}}
                  
\newcommand{\dx}{\mathrm{d}x}
\newcommand{\dy}{\mathrm{d}y}

\newcommand{\dt}{\mathrm{d}t}
\newcommand{\ds}{\mathrm{d}s}

\newcommand{\eps}{\varepsilon}
\renewcommand{\epsilon}{\varepsilon}
\renewcommand{\rho}{\varrho}

\DeclareMathOperator{\spt}{spt}

\DeclareMathOperator{\dist}{dist}

\DeclareMathOperator{\Div}{div}
\DeclareMathOperator{\ca}{cap}

\newcommand{\wto}{\rightharpoondown}
\newcommand{\wsto}{\overset{\raisebox{-1ex}{\scriptsize $*$}}{\rightharpoondown}}

\def\Xint#1{\mathchoice
    {\XXint\displaystyle\textstyle{#1}}%
    {\XXint\textstyle\scriptstyle{#1}}%
    {\XXint\scriptstyle\scriptscriptstyle{#1}}%
    {\XXint\scriptscriptstyle\scriptscriptstyle{#1}}%
    \!\int}
\def\XXint#1#2#3{\setbox0=\hbox{$#1{#2#3}{\int}$}
    \vcenter{\hbox{$#2#3$}}\kern-0.5\wd0}
\def\bint{\Xint-}
\def\dashint{\Xint{\raise4pt\hbox to7pt{\hrulefill}}}

\def\XXiint#1#2#3{\setbox0=\hbox{$#1{#2#3}{\iint}$}
    \vcenter{\hbox{$#2#3$}}\kern-0.5\wd0}

\newcommand{\power}[2]{\boldsymbol{#1^{\mbox{\unboldmath{\scriptsize$#2$}}}}}


\begin{document}
\title[Existence for doubly nonlinear systems in nondecreasing domains]{Existence of variational solutions to doubly nonlinear systems in nondecreasing domains}
\date{\today}

\author[L.~Schätzler]{Leah Schätzler}
\address{Leah Sch\"atzler\\
Department of Mathematics and Systems Analysis, Aalto University\\
P.O.~Box 11100, FI-00076 Aalto, Finland}
\email{ext-leah.schatzler@aalto.fi}

\author[C.~Scheven]{Christoph Scheven}
\address{Christoph Scheven\\
Fakultät für Mathematik, Universität Duisburg-Essen\\
Thea-Leymann-Str.~9, 45127 Essen, Germany}
\email{christoph.scheven@uni-due.de}

\author[J.~Siltakoski]{Jarkko Siltakoski}
\address{Jarkko Siltakoski\\
Department of Mathematics and Statistics, University of Helsinki, P.O.
Box 68 (Gustaf Hallstr\"omin katu 2b), Finland}
\email{jarkko.siltakoski@helsinki.fi}

\author[C.~Stanko]{Calvin Stanko}
\address{Calvin Stanko\\
Fachbereich Mathematik, Paris-Lodron Universit\"at Salzburg\\
Hellbrunner Str.~34, 5020 Salzburg, Austria}
\email{calvin.stanko@plus.ac.at}

\begin{abstract}
For $q \in (0, \infty)$, we consider the Cauchy--Dirichlet problem to doubly nonlinear systems of the form
\begin{align*}
    \partial_t \big( |u|^{q-1}u \big) - \operatorname{div} \big( D_\xi f(x,u,Du) \big) = - D_u f(x,u,Du)
\end{align*}
in a bounded noncylindrical domain $E \subset \mathds{R}^{n+1}$.
We assume that $x \mapsto f(x,u,\xi)$ is integrable, that $(u,\xi) \mapsto f(x,u,\xi)$ is convex, and that $f$ satisfies a $p$-coercivity condition for some $p \in (1,\infty)$.
However, we do not impose any specific growth condition from above on $f$.
For nondecreasing domains that merely satisfy $\mathcal{L}^{n+1}(\partial E) = 0$, we prove the existence of variational solutions $u \in C^{0}([0,T];L^{q+1}(E,\mathds{R}^N))$ via a nonlinear version of the method of minimizing movements.
Moreover, under additional assumptions on $E$ and a $p$-growth condition on $f$, we show that $|u|^{q-1}u$ admits a weak time derivative in the dual $(V^{p,0}(E))^{\prime}$ of the subspace $V^{p,0}(E) \subset L^p(0,T;W^{1,p}(\Omega,\mathds{R}^N))$ that encodes zero boundary values.
\end{abstract}

%***************************************************************************

\maketitle

\section{Introduction}
The main objective of the present paper is to prove the existence of variational solutions $u \colon E \to \mathds{R}^{N}$, $N \geq 1$, of doubly nonlinear systems, where $E \subset \R^n \times [0,T]$ for $n \geq 2$ and some $T>0$ is a noncylindrical domain. Here, we will focus on domains that are nondecreasing with respect to time (see \eqref{nondecreasing_condition} for the precise definition), whereas general domains will be treated in an upcoming paper. Before we specify in detail the more general systems considered in this paper, we briefly discuss the prototype equation, which is given by 
\begin{align}\label{prototype equation}
    \partial_t (|u|^{q-1}u) - \Div( \vert Du \vert^{p-2} Du ) = 0
		& \text{ in } E
\end{align}
with parameters $q \in (0,\infty)$ and $p \in (1, \infty)$. For the special choice of $q=1$ and $p=2$, this reduces to the heat equation. In the case of fixed $q=1$ and arbitrary $p \in (1, \infty)$ we are dealing with the parabolic $p$-Laplace equation. Vice versa, for fixed $p=2$ and arbitrary $q \in (0, \infty)$, \eqref{prototype equation} is the porous medium equation (also called fast diffusion equation for $q>1$).
More generally, for $p-1>q$, \eqref{prototype equation} is called slow diffusion equation, and for $p-1<q$ it is called fast diffusion equation \cite{Ivanov}. This is represented in the behaviour of solutions. Slow diffusion equations allow compactly supported solutions and perturbations propagate with finite speed, while for fast diffusion equations, perturbations propagate with infinite speed, which prevents solutions with compact support.

Furthermore, parabolic systems also have a wide range of applications such as fluid dynamics, cf.~\cite{Krechetnikov, Shelley}, especially the modeling of glacier formations, solidification and crystal
growth, as well as the Stefan problem, cf.~\cite{Stefan}, the Skorokhod problem, cf.~\cite{Skorokhod}, soil science, filtration, cf.~\cite{Aronson, Barenblatt1, Barenblatt2, Ladyzenskaja, Showalter_Walkington}, and mathematical biology, where reaction-diffusion equations are used for modeling pattern formations, cf.~\cite{Crampin_Gaffney_Maini, Murray}.
For parabolic equation in noncylindrical domains exhibiting stochastic properties, we refer to \cite{LS_84, Burdzy-Chen-Sylvester-1, Burdzy-Chen-Sylvester-2}. Other prominent examples of problems on time varying domains arise in fluid and quantum mechanics, cf.~\cite{FZ_2010, BBFSV13, BBFSV08}. A more detailed overview can be found in~\cite{Knobloch_Krechetnikov}.

We also want to give a brief insight on previous results. 
First, we will have a look on doubly nonlinear equations in cylindrical domains. The first doubly nonlinear equations were analyzed in \cite{Grange_Mignot} by Grange and Mignot, as well as in \cite{Alt_Luckhaus} by Alt and Luckhaus, who proved the existence of weak solutions for the equation 
\begin{align*}
\partial_{t} b(u) - \Div(a(b(u),Du))=f(b(u)).
\end{align*}
Here, $b$ is the gradient of a convex $C^{1}$-function satisfying $b(0)=0$ and $a(z, \xi)$ is a continuous and elliptic function in $(z,\xi)$ satisfying a $(p-1)$-growth condition in the second variable. The authors split the lateral boundary $\Omega_{T}$ into two complementary parts, where the first satisfies a Neumann condition and the second a Dirichlet condition. Furthermore, in their proofs the time derivative $\partial_{t} b(u)$ was replaced by the difference quotient $\Delta_{-h}b(u)$, where $h$ is fixed, and a Galerkin method was applied to the elliptic problems. The results in \cite{Alt_Luckhaus} and \cite{Grange_Mignot} were generalized to higher order doubly nonlinear equations on unbounded domains by Bernis in \cite{Bernis}.

Further, Akagi and Stefanelli considered in \cite{Akagi_Stefanelli} a Cauchy-Dirichlet problem for an equation of the form 
\begin{align*}
\partial_{t} b(u) - \Div(a(Du)) \ni f
\end{align*}
and with Dirichlet boundary values, where $b \subset \mathds{R} \times \mathds{R}$ and $a \subset \mathds{R}^{n} \times \mathds{R}^{n}$ are maximal monotone graphs satisfying polynomial growth conditions. They proved the existence for the dual formulation 
\begin{align*}
-\Div(a(Db^{-1}(v)) \ni f - \partial_{t} v 
\end{align*}
by means of the so called Weighted Energy Dissipation Functional approach, which is also known as Elliptic Regularization. 
Moreover, several existence results for doubly nonlinear equations have been proven via regularization techniques and a~priori estimates by Ivanov, Mkrtychyan and Jäger in \cite{Ivanov,Ivanov_Mkrtychyan, Ivanov_Mkrtychyan_Jaeger}. For $p \in (1,\infty)$ the Dirichlet boundary values were considered in the space $W^{1,p}(\Omega_{T}) \cap L^{\infty}(\Omega_T)$.

Later Bögelein, Duzaar, Marcellini and Scheven showed in \cite{BDMS18-2} that there exist non-negative variational solutions to the Cauchy--Dirichlet problem with time-independent boundary values to a general doubly nonlinear equation 
\begin{align*}
\partial_t b(u) - \Div( D_\xi f(x,u,Du) ) = - D_u f(x,u,Du),
\end{align*}
where the nonlinearity $b$ is an increasing, piecewise $C^{1}$ function satisfying certain growth conditions, and $f$ merely satisfies convexity and coercivity assumption, but no specific growth conditions. This allows $f$ to have an exponential or non-standard $(p,q)$-growth, where $1<p<q<\infty$.
Furthermore, assuming that $b=|u|^{q-1}u$, and that $f$ is independent of $x$ and $u$ and fulfills a standard $p$-growth condition, Schätzler \cite{Schaetzler} was able to deal with time-dependent Dirichlet boundary values $g \in L^p(0,T;W^{1,p}(\Omega,\R^N))$ with $g(0) \in L^{q+1}(\Omega)$ and $\partial_t g \in L^1(0,T;L^{q+1}(\Omega))$.
For related obstacle problems, see also \cite{Schaetzler-obstacle-deg,Schaetzler-obstacle-sing}.

Moving on to equations in noncylindrical domains, in other words domains which are allowed to grow or shrink with respect to time, we have a first look on linear equations. A significant amount of research has been dedicated to linear equations in noncylindrical domains, cf.~\cite{Acquistapace-Terreni, Baiocchi, BHL1997, Cannarsa_Prato_Zolesio, Lions1957, Lumer-Schnaubelt-2001}. The method of minimizing movements, which we also use in the present article, was first implemented by Gianazza and Savaré~\cite{Gianazza_Savare} for domains, which do not decrease in time. Later, Savaré \cite{Savare} considered uniformly $C^{1,1}$-domains with a shrinking speed bounded by a distance function fulfilling a Lipschitz condition. We point out that the method of minimizing movements was first suggested by De Giorgi in \cite{De_Giorgi} as a way to show existence of weak solutions for linear parabolic equations in noncylindrical domains. This assumption was changed to a Hölder condition by Bonaccorsi and Guatteri in \cite{Bonaccorsi_Guatteri}. Furthermore, shrinking domains tend to be more difficult to handle in general, since the boundary value problem could become overdetermined if the domain shrinks too fast.

Nonlinear equations in time-dependent domains, where each time slice is Lip\-schitz and a regular deformation of another, were first analyzed by Paronetto, cf. \cite{Paronetto}.
The existence of weak solutions to equations of parabolic $p$-Laplace type in bounded open sets $E$ with a Lipschitz boundary in space-time was shown by Calvo, Novaga and Orlandi \cite{Calvo_Novaga_Orlandi}.
Under stronger conditions on $E$ (that are e.g.~satisfied for purely expanding domains), the authors obtain uniqueness.

The gradient flow associated to an integral functional in a noncylindrical bounded domain $E$ with systems of the form 
\begin{align}\label{eq:PDE_with_linear_time_derivative_from_BDSS18_paper}
    \partial_t u - \Div( D_\xi f(x,u,Du) ) = - D_u f(x,u,Du) \text{ on } E
\end{align}
was studied by Bögelein, Duzaar, Scheven and Singer in \cite{BDSS18}. They proved existence of variational solutions under the only assumption that the variational integrand $f$ is convex and coercive with respect to the $W^{1,p}$-norm for some $p>1$, and that the noncylindrical domain $E$ satisfies the weak regularity condition $\L^{n+1}(\partial E) = 0$.
However, the integrand $f$ does not fulfil a specific growth condition.
The authors use different methods to deal with nondecreasing and general domains.
In nondecreasing domains, they prove the existence of variational solutions via the method of minimizing movements.
Moreover, they prove that these solutions are unique, continuous as maps from $[0,T)$ into $L^2(\Omega,\R^N)$, and possess a weak time derivative $\partial_t u$ in the space $L^{2}(E,\R^N)$.
If $f$ additionally satisfies a standard $p$-growth condition, $u$ admits a time derivative in the dual of the parabolic function space encoding the boundary values.

Our goal is to show the existence of variational solutions in noncylindrical domains.
We will combine the techniques and methods used in \cite{BDMS18-2} and \cite{BDSS18}. However, we weakened certain assumptions and generalized some aspects of the problem. In contrast to \cite{BDMS18-2} the solutions are not restricted to be non-negative. Furthermore, instead of the space-time cylinder $\Omega_{T}$ we will consider a noncylindrical domain, similar to \cite{BDSS18}. While the assumptions on $f$ are analogous, cf. \eqref{eq:integrand}, our primary contribution lies in the analysis of a generalized version of Equation \eqref{eq:PDE_with_linear_time_derivative_from_BDSS18_paper}, where $\partial_t u$ is replaced by the nonlinear term $\partial_t \big( |u|^{q-1} u \big)$, which is a special case of $\partial_t b(u)$ in \cite{BDMS18-2}. Moreover, unlike the work of Bögelein, Duzaar, Scheven and Singer, which dealt with zero lateral boundary conditions in \cite{BDSS18}, we impose time-independent lateral boundary values, similar to \cite{BDMS18-2}. Nevertheless, we expect that our methods can be used for time-dependent boundary values.
We note that in the case of the doubly nonlinear systems considered here, the question of uniqueness is much more delicate.
In particular, the inclusion of coefficients is already difficult in the case of cylindrical domains.
However, for filtration equations that include the prototype porous medium equation as a special case, uniqueness can be proven by testing with a Ole\u{\i}nik type test function, see \cite[Theorem 5.13]{Vazquez}.
In noncylindrical domains, Abdulla \cite{Abdulla-uniqueness} established uniqueness for non-negative solutions to the prototype porous medium equation with non-negative continuous boundary values.
If $E$ is a finite union of cylinders, the assumptions on the regularity of the lateral boundary were weakened by Björn, Björn, Gianazza and Siljander \cite{Bjoern-etal-regular-points-PME}.
Moreover, doubly nonlinear equations that are independent of $(x,t)$ were investigated by Otto \cite{Otto} in the cylindrical setting.
He proved uniqueness of solutions to the Dirichlet problem with time independent boundary values via an $L^1$-contraction principle.

\section{Setting and main results}
\subsection{Setting}
First, we  give details on our setting and notation.
Let $\Omega \subset \R^n$, $n \geq 2$, be open and bounded, $0<T<\infty$ and define the space-time cylinder $\Omega_T := \Omega \times [0,T)$.
Further, consider a relatively open noncylindrical domain $E \subset \Omega_T$ merely fulfilling the regularity assumption
\begin{equation}
	\L^{n+1}(\partial E) = 0.
	\label{eq:boundary_E_zero_measure}
\end{equation}
Denoting the time slice with fixed time $t \in [0,T)$ of $E$ by
\begin{align*}
	E^t := \{x \in \R^n : (x,t) \in E\} \subset \R^n,
\end{align*}
we find that
\begin{align*}
	E = \bigcup_{t \in [0,T)} E^t \times \{t\}.
\end{align*}
We will mainly be concerned with domains which are nondecreasing in time in the sense that the time slices satisfy the condition
\begin{align}\label{nondecreasing_condition}
    E^s \subset E^t
	\quad \text{for any } 0 \leq s \leq t < T.
\end{align}

For a parameter $q \in (0,\infty)$, an initial datum $u_o$, and boundary condition $u_{\ast}$ we are concerned with the Cauchy--Dirichlet problem
\begin{equation}
	\left\{
	\begin{array}{cl}
		\partial_t (|u|^{q-1}u) - \Div( D_\xi f(x,u,Du) ) = - D_u f(x,u,Du)
		& \text{in } E, \\[5pt]
		u=u_{\ast}
		&\text{on } \partial_\mathrm{lat} E, \\[5pt]
		u(\cdot,0) = u_o
		&\text{in } E^0,
	\end{array}
	\right.
	\label{eq:system}
\end{equation}
where $u \colon E \to \R^N$, $N \geq 1$, and the lateral boundary of $E$ is defined by
$$
    \partial_\mathrm{lat} E :=\bigcup_{t \in [0,T)} \partial E^t \times \{t\}.
$$
Here, we assume that for $p \in (1,\infty)$ and a constant $\nu > 0$, the integrand $f \colon \Omega \times \R^N \times \R^{Nn} \to [0,\infty)$ satisfies the conditions
\begin{equation}
	\left\{
	\begin{array}{l}
		\mbox{$x \mapsto f(x,u,\xi) \in L^1(\Omega)$ for any $(u,\xi) \in \R^N \times \R^{Nn}$,} \\[5pt]
		\mbox{$(u,\xi) \mapsto f(x,u,\xi)$ is convex for a.e.~$x \in \Omega$,} \\[5pt]
		\mbox{$\nu |\xi|^p \leq f(x,u,\xi)$ for a.e.~$x \in \Omega$ and all $(u,\xi) \in \R^N \times \R^{Nn}$.} 
	\end{array}
	\right.
	\label{eq:integrand}
\end{equation}
We point out that we do not require a growth assumption from above. However, the assumptions \eqref{eq:integrand}$_1$ and \eqref{eq:integrand}$_2$ imply the following result, which can be seen as a mild growth property. Its proof is analogous to \cite[Lemma 2.3]{Schaetzler-Siltakoski}.
It shows that the integrability in $x$ is uniform with respect to $(u,\xi)$ in any bounded set in $\R^N \times \R^{Nn}$.
\begin{lemma}
There exists a measurable map $g \colon \Omega \times [0,\infty) \to [0,\infty)$ such that $x \mapsto g(x,M) \in L^1(\Omega)$ for any $M \geq 0$ and
\begin{align}\label{ineq:integrand_f_bound}
    0 \leq f(x,u,\xi) \leq g(x,M)
\end{align}
	holds true for a.e.~$x \in \Omega$ and all $(u,\xi) \in \R^N \times \R^{Nn}$ such that $\max\{|u|,|\xi|\} \leq M$.
\end{lemma}

Next, we impose the following conditions on the initial and lateral boundary values $u_o$ and $u_{\ast}$:
\begin{equation}
	\left\{
	\begin{array}{l}
		\mbox{$u_o \in L^{q+1}(\Omega, \mathds{R}^{N}$),} \\[5pt]
		\mbox{$u_{\ast} \in L^{q+1}(\Omega, \mathds{R}^{N}) \cap W^{1,p}(\Omega, \mathds{R}^{N})$,} \\[5pt]
		\mbox{$0 \leq \int_{\Omega} f(x, u_{\ast} +\phi , D(u_{\ast} +\phi)) \,\dx < \infty,  \quad \forall \phi \in C_{0}^{\infty}(\Omega, \mathds{R}^{N})$.} 
	\end{array}
	\right.
	\label{compatibility_condition_for_u_o_and_u_ast}
\end{equation}

Now, we introduce appropriate function spaces for variational solutions and comparison maps.
First, on the level of the time slices we define the space
\begin{align*}
	V_t :=
	\big\{ v \in W^{1,p}_{u_{\ast}}(\Omega,\R^N) : v=u_{\ast} \text{ a.e.~in } \Omega \setminus E^t \big\}.
\end{align*}
Further, we consider the parabolic function spaces
\begin{align*}
	V^p(E) :=
	\big\{ v \in L^p\big(0,T;W^{1,p}_{u_\ast}(\Omega,\R^N)\big) : v(t) \in V_t \text{ for a.e.~} t \in [0,T) \big\},
\end{align*}
and
\begin{align*}
	V^p_q(E) :=
	V^p(E) \cap L^{q+1}(\Omega_T,\R^N), 
\end{align*}
equipped with the norms
\begin{align*}
    \| v \|_{V^p(E)}
	:=
	\| Dv \|_{L^p(\Omega_T,\R^{Nn})} + \| v \|_{L^p(\Omega_T,\R^{N})},
\end{align*}
and
\begin{align*}
    \| v \|_{V^p_q(E)}
	:=
	\| Dv \|_{L^p(\Omega_T,\R^{Nn})} + \| v \|_{L^{q+1}(\Omega_T,\R^N)},
\end{align*}
respectively.

For $u \in \R^N$ we use the abbreviation
\begin{align*}
    \power{u}{q}
	:=
	|u|^{q-1} u
\end{align*}
and define the boundary term by
\begin{align*}
	\b[u,v]
	&:=
	\tfrac{1}{q+1} |v|^{q+1} - \tfrac{1}{q+1} |u|^{q+1} - \power{u}{q} \cdot (v-u) \\
	&=
	\tfrac{1}{q+1} |v|^{q+1} + \tfrac{q}{q+1} |u|^{q+1} - \power{u}{q} \cdot v.
\end{align*}

At this stage, we define variational solutions to the Cauchy-Dirichlet problem \eqref{eq:system}.
Similar as in \cite{BDMS18-1}, one can formally derive the variational inequality from ~\eqref{eq:system}$_{1}$.
This variational approach was originally introduced by Lichnewsky and Teman \cite{Lichnewsky_Temam} and developed in \cite{BDM-15, BDMS18-2, BDM13} by Bögelein, Duzaar, Marcellini and Scheven in order to prove existence results for a wide range of parabolic equations.

\begin{definition}[Variational solutions]\label{Definition:variational_solution}
Suppose that $E \subset \Omega_T$ is a relatively open noncylindrical domain, that $f \colon \Omega \times \R^N \times \R^{Nn} \to [0,\infty)$ satisfies \eqref{eq:integrand}, and that $u_{\ast}, u_o \colon \Omega \to \mathds{R}^N$ satisfy \eqref{compatibility_condition_for_u_o_and_u_ast}.
A measurable map
\begin{align}\label{measurable_map_var_ineq}
    u \in L^\infty(0,T;L^{q+1}(\Omega,\R^N)) \cap V^p_q(E)
\end{align}
is called a \emph{variational solution} to \eqref{eq:system} if and only if the variational inequality
\begin{align}
	\iint_{E\cap \Omega_\tau} &f(x,u,Du) \,\dx\dt
	\nonumber \\&\leq
	\iint_{E\cap \Omega_\tau} f(x,v,Dv) \,\dx\dt
	+\iint_{E\cap \Omega_\tau} \partial_t v \cdot (\power{v}{q} - \power{u}{q}) \,\dx\dt
    \label{eq:variational_inequality} \\
	&\phantom{=}
	-\int_{E^\tau}	\b[u(\tau),v(\tau)] \,\dx
	+\int_{E^0} \b[u_o,v(0)] \,\dx
    \nonumber
\end{align}
holds true for a.e.~$\tau \in (0,T)$ and any comparison map $v \in V^p_q(E)$ with time derivative $\partial_t v \in L^{q+1}(\Omega_T,\R^N)$.
\end{definition}

\subsection{Main results}
First, we establish the existence of variational solutions in nondecreasing domains.
\begin{theorem}
\label{thm:existence_in_nondecreasing_domains}
    Assume that $E \subset \Omega_{T}$ is a relatively open noncylindrical domain satisfying \eqref{eq:boundary_E_zero_measure} and \eqref{nondecreasing_condition}, and that $f\colon \Omega \times \mathds{R}^{N} \times \mathds{R}^{Nn} \to [0,\infty)$ satisfies \eqref{eq:integrand}.
    Then, for any boundary value $u_{\ast}$ and initial datum $u_o$ satisfying \eqref{compatibility_condition_for_u_o_and_u_ast}, 
    there exists a variational solution
    \begin{align*}
        u \in C^{0} \big([0,T);L^{q+1}(\Omega,\R^N) \big) \cap V^p_q(E)
    \end{align*}
in the sense of Definition~\ref{Definition:variational_solution}.
    Moreover, we have that $\partial_t \power{u}{\frac{q+1}{2}} \in L^2(\Omega \times (\epsilon,T),\R^N)$ and $\partial_t \power{u}{q{+}1} \in L^1(\Omega \times (\epsilon,T),\R^N)$ for any $\epsilon > 0$.
\end{theorem}

\begin{remark}
The argument in Section \ref{Section:Convergence_almost_everywhere_non_decreasing} shows that $\partial_t \power{u}{\frac{q+1}{2}} \in L^2(\Omega_T,\R^N)$ and $\partial_t \power{u}{q{+}1} \in L^1(\Omega_T,\R^N)$, provided that in addition to the assumptions of Theorem~\ref{thm:existence_in_nondecreasing_domains} we suppose that
$$
    u_o \in V_0
    \qquad \text{and}\qquad
    \int_\Omega f(x,u_o,Du_o) \,\dx < \infty.
$$
\end{remark}

Next, we define the function spaces
$$
    V^{p,0}(E)
    :=
    \{v \in L^p(0,T;W^{1,p}_0(\Omega,\R^N)) : v(t) \equiv 0 \text{ in } \Omega \setminus E^t \text{ for a.e. } t \in [0,T] \},
$$
and
\begin{align*}
    V_{q}^{p,0}(E):= V^{p,0}(E) \cap L^{q+1}(\Omega_{T}, \mathds{R}^{N}),
\end{align*}
equipped with the norms $\| v \|_{V^{p}(E)}$ and $\| v \|_{V^{p}_q(E)}$, respectively.
Under some additional conditions on $E$ we show that $\power{u}{q}$ possesses a distributional time derivative in the dual space $(V^{p,0}(E))^{\prime}$.
Indeed, we assume that for all $t \in [0,T)$, the complement $\Omega \setminus E^{t}$ satisfies a measure density condition, i.e., that there exists a constant $\delta>0$ such that there holds
\begin{align}\label{ineq:lower_measure_bound}
    \L^{n} \big(\big(\mathds{R}^{n} \setminus E^{t}\big) \cap B_{r}(x_o) \big) \ge \delta \L^{n}(B_{r}(x_o)) \quad \forall x_o \in \mathds{R}^{N} \setminus E^{t} \text{ and } \forall r >0.
\end{align}
Further, in addition to \eqref{eq:integrand}, we require that $f$ satisfies the $p$-growth condition
\begin{align}\label{adjusted standard-coercivity and p-growth condition}
    f(x,u,\xi) \le L \big(\vert \xi \vert^{p} + \vert u \vert^{p} + G(x) \big),
\end{align}
for a.e.~$x \in \Omega$ and all $(u, \xi) \in \mathds{R}^{N} \times \mathds{R}^{Nn}$, where $0 < \nu \le L$ and $G \in L^1(\Omega_{T})$.
Note that \eqref{adjusted standard-coercivity and p-growth condition} is a special case of \eqref{ineq:integrand_f_bound}.
In particular, under \eqref{adjusted standard-coercivity and p-growth condition}, the condition \eqref{compatibility_condition_for_u_o_and_u_ast}$_3$ holds for any $u_\ast \in W^{1,p}(\Omega,\R^N)$.
Moreover, adapting the proof of \cite[Lemma 2.1]{Marcellini}, we find that \eqref{eq:integrand} and \eqref{adjusted standard-coercivity and p-growth condition} imply the local Lipschitz condition
\begin{align}\label{ineq:Lipschitz_condition}
    \vert f(x,u_{1}, \xi_{1} ) - &f(x,u_{2}, \xi_{2} ) \vert \nonumber \\
    &\le c\Big[ (\vert \xi_{1} \vert + \vert \xi_{2} \vert + \vert u_{1} \vert + \vert u_{2} \vert)^{p-1} + \vert G(x) \vert^\frac{p-1}{p}  \Big](\vert u_{1} - u_{2} \vert + \vert \xi_{1} - \xi_{2} \vert )
\end{align} 
with a constant $c=c(p,L)$ for a.e.~$x \in \Omega$ and any $(u_{1}, \xi_{1}), (u_{2}, \xi_{2}) \in \mathds{R}^{N} \times \mathds{R}^{Nn}$.
Now, our result on the existence of a time derivative in the dual space reads as follows.

\begin{theorem}\label{thm:time_derivative_nondecreasing}
Assume that $E$ is a relatively open noncylindrical domain satisfying \eqref{nondecreasing_condition} and  \eqref{ineq:lower_measure_bound}, that the variational integrand $f \colon \Omega \times \mathds{R}^{N} \times \mathds{R}^{Nn} \to [0,\infty)$ fulfills \eqref{eq:integrand} and ~\eqref{adjusted standard-coercivity and p-growth condition}, and that $u_{\ast} \in W^{1,p}(\Omega, \mathds{R}^{N}) \cap L^{q+1}(\Omega, \mathds{R}^{N})$ and $u_o \in L^{q+1}(E^{0},\mathds{R}^{N})$.
Further, let $u$ be a variational solution to \eqref{eq:system} in the sense of Definition~\ref{Definition:variational_solution}.
Then, $\power{u}{q}$ admits a distributional time derivative
 \begin{align*}
     \partial_{t} \power{u}{q} \in 
     \big(V^{p,0}(E)\big)'
 \end{align*}
 with the estimate
\begin{align*}
     \Vert  \partial_{t} \power{u}{q} \Vert_{(V^{p,0}(E))^{\prime}} \le c(p,L) \Big[ \Vert  u \Vert_{V^{p}(E)}^{p} + \Vert G \Vert_{L^1(\Omega)}\Big]^{\frac{p-1}{p}}.
 \end{align*}
\end{theorem}

\subsection{Methods of proof}

The proof of the existence result from Theorem~\ref{thm:existence_in_nondecreasing_domains} is given in Section~\ref{sec:existence}. Our main approach employs a nonlinear version of the method of minimizing movements. For this we combine the methods from \cite{BDMS18-2}, where doubly nonlinear equations were treated, and \cite{BDSS18}, which was concerned with noncylindrical domains. First, we apply a time discretization in order to divide the time interval $(0,T)$ into smaller ones with length $h>0$. Then we iteratively solve elliptic problems on each time slice by solving a minimization problem, yielding a sequence of minimizing functions. Next, we reformulate the minimizing property related to the considered elliptic problems, in order to develop energy estimates and consequently some weak convergence results in the upcoming steps. Subsequently, these minimizers are merged together to a piecewise constant function $u_\ell$ in time. Afterwards, we prove that the limit map of $u_\ell$ admits the correct boundary values, relying on the previously developed convergence results and the assumption $\L^{n+1}(\partial E)=0$. In this step and the upcoming one we exploit the fact that we are treating nondecreasing domains. We proceed by obtaining the almost everywhere convergence of $u_\ell$, which enables us to show that the limit map satisfies the variational inequality. This proves the convergence of these approximating functions $u_\ell$ to a solution of the parabolic problem in the noncylindrical domain. 

Since we treat nondecreasing domains, we can directly establish the continuity in time for any variational solution, cf.~Section~\ref{sec:continuity in time}. This is achieved by testing the variational inequality with a mollification of $u$ with respect to time. The boundary values are preserved, since the mollification is defined as a weighted mean of values from previous time steps.
We deduce that the mollifications converge uniformly in time to the variational solution, which implies that the solution is continuous in time, with respect to the $L^{q+1}$-norm.

Finally, we give the proof of Theorem~\ref{thm:time_derivative_nondecreasing} in Section~\ref{Section:Time derivative in the dual space(nondecreasing)} following the approach in \cite[Section 4.2.3]{BDSS18}. We show that any variational solution is a so-called parabolic minimizer and then use the minimality condition to establish the existence of the time derivative $\partial_t (\vert u \vert^{q-1} u)$ in the dual space.

\section{Preliminaries}
\subsection{Technical lemmas}
We start this section by deriving estimates that allow us to deal with powers of $u$ and the boundary term $\b[u,v]$.
For the following lemma, we refer to \cite[Lemma 8.3]{Giusti}; see also \cite[Lemma 3.2]{BDKS20} for the statement.
\begin{lemma}\label{ineq:power_alpha_estimate_q_greater_0}
For any $\alpha > 0$ and $u, v \in \mathds{R}^{N}$, there exists a constant $c(\alpha)$ such that
\begin{align*}
    \tfrac{1}{c(\alpha)} \vert \power{v}{\alpha} - \power{u}{\alpha} \vert \le  ( \vert u \vert + \vert v \vert )^{\alpha -1} \vert v - u \vert \le c(\alpha) \vert \power{v}{\alpha} - \power{u}{\alpha} \vert.
\end{align*}
\end{lemma}

As a consequence of the preceding lemma, we obtain the following inequality; see also \cite[Lemma 3.3]{BDKS20}.
\begin{lemma}\label{ineq:power_alpha_estimate}
For any $\alpha > 1$ and $u, v \in \mathds{R}^{N}$, there exists a constant $c(\alpha)$ such that
\begin{align*}
    \vert v - u \vert^{\alpha} \le c(\alpha) \vert \power{v}{\alpha} - \power{u}{\alpha} \vert. 
\end{align*}
\end{lemma}

In turn, the preceding statement leads us to the following lemma.
\begin{lemma} \label{lem:technical_lemma_2}
For any $q \in (0,\infty)$ and $u,v \in \R^N$, we have that
\begin{align*}
	\vert u-v \vert^{q+1}
	\leq
	c(q) \Big(\vert u \vert^\frac{q+1}{2} + \vert v \vert^\frac{q+1}{2} \Big) \Big\vert \power{u}{\frac{q+1}{2}} - \power{v}{\frac{q+1}{2}} \Big\vert.
\end{align*}
\end{lemma}
\begin{proof}
By Lemma \ref{ineq:power_alpha_estimate} with $\alpha$ replaced by $q+1$, the triangle inequality and the reverse triangle inequality, we find that
\begin{align*}
	|u-v|^{q+1}
	&\leq
	c(q) \big| \boldsymbol{u}^{q+1} - \boldsymbol{v}^{q+1} \big| \\
	&=
	c(q) \Big| |u|^\frac{q+1}{2} \power{u}{\frac{q+1}{2}} - |v|^\frac{q+1}{2} \power{u}{\frac{q+1}{2}} + |v|^\frac{q+1}{2} \power{u}{\frac{q+1}{2}} - |v|^\frac{q+1}{2} \power{v}{\frac{q+1}{2}} \Big| \\
	&\leq
	c(q) \Big( |u|^\frac{q+1}{2} \Big| |u|^\frac{q+1}{2} - |v|^\frac{q+1}{2} \Big|
	+ |v|^\frac{q+1}{2} \Big| \power{u}{\frac{q+1}{2}} - \power{v}{\frac{q+1}{2}} \Big| \Big) \\
	&\leq
	c(q) \Big( |u|^\frac{q+1}{2} + |v|^\frac{q+1}{2} \Big) \Big| \power{u}{\frac{q+1}{2}} - \power{v}{\frac{q+1}{2}} \Big|.
	\qedhere
\end{align*}
\end{proof}

Next, we have the following lemma.
\begin{lemma}
\label{lem:technical_lemma}
For any $q \in (0,\infty)$ and $u,v \in \R^N$, we have that
\begin{align*}
	\Big| \power{u}{\frac{q+1}{2}} - \power{v}{\frac{q+1}{2}} \Big|^2
	\leq
	c(q) \b[u,v]
	\leq
	c(q) \big( \power{v}{q} - \power{u}{q} \big) \cdot (v-u).
\end{align*}
\end{lemma}
\begin{proof}
For the first inequality, we refer to \cite[Lemma 3.4]{BDKS20} with $p$ replaced by $q+1$.
Since $w \mapsto \frac{1}{q+1} |w|^{q+1} \in C^1(\R^N)$ is convex with $\frac{1}{q+1} D |w|^{q+1} = \power{w}{q}$, we have that
\begin{align*}
	\tfrac{1}{q+1} |v|^{q+1} - \tfrac{1}{q+1} |u|^{q+1}
	\leq
	\power{v}{q} \cdot (v-u).
\end{align*}
Recalling the definition of $\b[u,v]$, this implies the claim.
\end{proof}

For the proof of the subsequent estimates, cf.~\cite[Lemma 2.1]{Schaetzler}; see also \cite[Lemma 2.4]{BDMS18-2} for the special case $N=1$.
\begin{lemma}
\label{lem:b_and_abs_u_estimate}
    Let $u,v \in \mathds{R}^{N}$. Then, we have that
    \begin{equation*}
        \tfrac{1}{q+1} \vert v \vert^{q+1} \le 2 \b[u,v] +2^{2+\frac{1}{q}} \tfrac{q}{q+1}\vert u \vert^{q+1}.
    \end{equation*}
\end{lemma}

Next, we state the following convergence result.
\begin{lemma}\label{lem:strong_L^q+1_convergence}
    Let $(u_{k})_{k \in \mathds{N}} \subset L^{q+1}(\Omega, \mathds{R}^{N})$  be a sequence and $u_\infty \in L^{q+1}(\Omega, \mathds{R}^{N})$, such that $\boldsymbol{u}_{\boldsymbol{k}}^{q+1} \to \boldsymbol{u}_{\boldsymbol{\infty}}^{q+1}$ in $L^{1}(\Omega, \mathds{R}^{N})$, for $k \to \infty$.
    Then, we have that $u_{k} \to u_\infty$ in $L^{q+1}(\Omega, \mathds{R}^{N})$.
    Moreover, if $u_{k} \to u_\infty$ in $L^{q+1}(\Omega, \mathds{R}^{N})$, for $k \to \infty$, then we have the convergence $\power{u_{k}}{q} \to \power{u_\infty}{q}$ in $L^{\frac{q+1}{q}}(\Omega, \mathds{R}^{N})$.
\end{lemma}
\begin{proof}
For the case $q>1$, we refer to \cite[Lemma~2.4]{Schaetzler}.
Thus, it remains to consider the case $0<q<1$.
    The first part of the lemma can be deduced by applying Lemma ~\ref{ineq:power_alpha_estimate} with $\alpha = q+1$, which yields
    \begin{align*}
        \int_{\Omega} \vert u_{k} - u_\infty \vert^{q+1} \,\dx \le c(q) \int_{\Omega} \vert \boldsymbol{u}_{\boldsymbol{k}}^{q+1} - \boldsymbol{u}_{\boldsymbol{\infty}}^{q+1} \vert \,\dx.
    \end{align*}
    For the second part of the lemma, we apply Lemma \ref{ineq:power_alpha_estimate} with $\alpha = \tfrac{1}{q}>1$. Hence, we have that
    \begin{align*}
        \int_{\Omega} \big\vert \power{u_{k}}{q} - \power{u_\infty}{q} \big\vert^{\frac{q+1}{q}} \,\dx
        &\le
        c(q) \int_{\Omega} \big\vert (\power{u_{k}}{q})^{\frac{1}{q}} - (\power{u_\infty}{q})^{\frac{1}{q}} \big\vert^{q+1} \,\dx \\
        &=
        c(q)\int_{\Omega} \vert u_{k} - u_\infty \vert^{q+1} \,\dx \to 0
    \end{align*}
    as $k \to \infty$, which concludes the proof.
\end{proof}

We proceed with a Hardy inequality, which holds under a weaker assumption than the measure density condition~\eqref{ineq:lower_measure_bound}.
More precisely, denoting the variational $p$-capacity as defined in \cite[Def.~5.32]{KiLeVa} by $\ca_p$, we rely on the following concept.
\begin{definition} \label{def:p-fat}
We say that a set $A \subset \R^n$ is uniformly $p$-fat, if there exists a constant $\alpha > 0$ such that
$$
\ca_p\big(A \cap \overline{B_\rho(x)}, B_{2\rho}(x)\big) \geq \alpha \ca_p\big(\overline{B_\rho(x)}, B_{2\rho}(x)\big) 
$$
holds true for every $x \in A$ and $\rho > 0$.
\end{definition}

Since Hardy's inequality only applies to functions with $u(t)\in W^{1,p}_0(E^t)$ for a.e.~$t\in[0,T)$, we consider the subspace of $V^{p,0}(E)$ given by
\begin{equation}
\mathcal{V}^{p,0}(E):=\big\{u\in V^{p,0}(E)\colon u(t)\in
W^{1,p}_0(E^t,\R^N) \mbox{ for a.e.~$t\in[0,T)$} \big\}.
\label{eq:curly-Vp0}
\end{equation}

\begin{remark}\label{rem:V=V}
The measure density condition~\eqref{ineq:lower_measure_bound} with a constant $\delta>0$ implies that $\R^n\setminus E^t$ is uniformly $p$-fat with a parameter $\alpha=\alpha(n,p,\delta)>0$, see, e.g. \cite[Example 6.18]{KiLeVa}.
Further, under the  measure density condition we have that $\mathcal{V}^{p,0}_q(E)=V^{p,0}_q(E)$, cf.~\cite[Lemma~3.1]{BDSS18}. 
\end{remark}

For the proof of the following lemma we refer to \cite{KiMa} or \cite[Theorem 6.25]{KiLeVa}.
\begin{lemma}[Hardy's inequality]\label{lem:Hardy_inequality}
  Let $u \in \mathcal{V}^{p,0}(E)$ and assume that $\R^n\setminus E^t$
  is uniformly $p$-fat with a parameter $\alpha>0$ for a.e.~$t\in[0,T)$.  
  Then there exists a constant $c(n, p, \alpha)>0$
  such that the Hardy inequality
    \begin{align}\label{ineq:Hardy_inequality}
        \int_{E^t} \bigg( \frac{\vert u(x,t) \vert}{ \dist(x, \partial E^t)  } \Bigg)^{p} \,\dx
        \le
        c  \int_{E^t} \vert Du(x,t) \vert^{p} \,\dx 
    \end{align}
    holds for a.e.~$t \in [0,T)$.
\end{lemma}

Next, we introduce the finite difference quotient $\Delta_{h}f$ for a function $f$ with respect to time, with $h \in \mathds{R} \setminus \{0\}$, by
\begin{align*}
    \Delta_{h}f(t):= \tfrac{1}{h}(f(t+h)-f(t))
\end{align*}
and state the following integration by parts formula.
We refer to~\cite[Lemma 2.11]{Schaetzler}, see also \cite[Lemma 2.10]{BDMS18-2} for the case $N=1$ and non-negative solutions.
\begin{lemma}\label{lem:finite_integration_by_parts_formula}
    Let $h \in (0,1]$, and $u,v \in L^{q+1}(\Omega \times (-h,T+h), \mathds{R}^{N})$. Then, the following finite integration by parts formula
    \begin{align*}
        \iint_{\Omega_{T}} \Delta_{-h} \power{u}{q} \cdot (v-u) \,\dx\dt \le& \iint_{\Omega_{T}} \Delta_{h} v \cdot (\power{v}{q}-\power{u}{q}) \,\dx\dt \nonumber \\
        &-\tfrac{1}{h} \iint_{\Omega \times (T-h,T)} \b[u(t),v(t+h)] \,\dx\dt \nonumber\\
        &+\tfrac{1}{h} \iint_{\Omega \times (-h,0)} \b[u(t),v(t)] \,\dx\dt +\delta_{1}(h)+\delta_{2}(h),
    \end{align*}   
    holds, where $\delta_{1}(h)$ and $\delta_{2}(h)$ are defined by
    \begin{align*}
    \delta_{1}(h)&:= \tfrac{1}{h} \iint_{\Omega_{T}} \b[v(t),v(t+h)] \,\dx\dt, \\
    \delta_{2}(h)&:=  \iint_{\Omega \times (- h, 0)} \Delta_{h}v \cdot (\power{v}{q}(t+h) - \power{u}{q}(t)) \,\dx\dt. 
\end{align*}
If $\partial_{t}v \in L^{q+1}(\Omega \times (-h_{0}, T+h_{0}),\mathds{R}^{N})$ for some $h_{0} > 0$, then
\begin{align*}
    \lim_{h \downarrow 0} \delta_{1}(h)=0
    \quad\text{and}\quad
    \lim_{h \downarrow 0} \delta_{2}(h)=0.
\end{align*}
\end{lemma}

We conclude this section by the following compactness result for signed or vector-valued functions.
It has been proven in \cite[Proposition 3.1]{BDMS18-2} in the case of non-negative functions and in \cite[Lemma 2.12]{Schaetzler} for vector-valued or signed functions.
\begin{lemma}
\label{lem:compactness}
Let $\Omega \subset \R^n$ be a bounded domain, $0<T<\infty$, $p>1$ and $q \in (0,\infty)$.
Assume that for $h_\ell := \frac{T}{\ell}$ for some $\ell \in \N$ maps $u_\ell \colon \Omega \times (-h_\ell, T] \to \R^N$ are defined by
\begin{align*}
	u_\ell(\cdot,t) := u_{\ell,i}
	\quad \text{for $t \in ((i-1) h_\ell, ih_\ell]$ with $i \in \{0, \ldots, \ell\}$,}
\end{align*}
where $u_{\ell,i} \in L^{q+1}(\Omega,\R^N) \cap W^{1,p}(\Omega,\R^N)$.
Further, suppose that $u_\ell \wto u$ weakly in $L^p(0,T;W^{1,p}(\Omega,\R^N))$ as $\ell \to \infty$.
Moreover, assume that there exists a constant $C$ such that the energy estimate
\begin{align*}
	\max_{i \in \{0, \ldots, \ell\}}
	\bigg( \int_\Omega |u_{\ell,i}|^{q+1} \,\dx
	+ \int_\Omega |Du_{\ell,i}|^p \,\dx \bigg)
	\leq C
\end{align*}
and the continuity estimate
\begin{align*}
	\tfrac{1}{h_\ell} \sum_{i=1}^\ell \int_\Omega \Big| \power{u_{\ell,i}}{\frac{q+1}{2}} - \power{u_{\ell,i-1}}{\frac{q+1}{2}} \Big|^2 \,\dx
	\leq C
\end{align*}
are satisfied for all $\ell \in \N$.
Then, there exists a subsequence $\mathfrak{K} \subset \N$ such that there holds
\begin{align}\label{lem:strong_L1_conv}
	\left\{
	\begin{array}{ll}
		\power{u_\ell}{\frac{q+1}{2}} \to \power{u}{\frac{q+1}{2}}
		& \text{strongly in } L^1(\Omega_T,\R^N), \\[5pt]
		u_\ell \to u
		&\text{a.e.~in } \Omega_T,
	\end{array}
	\right.
\end{align}
in the limit $\mathfrak{K}\ni \ell\to\infty$.         
\end{lemma}

\subsection{Mollification in time}
We will use the following regularization with respect to time introduced by Landes \cite{Landes}.
Given $h>0$ and a function $v$, the mollification $[v]_h$ is constructed to satisfy the ordinary differential equation
\begin{equation}
	\partial_t [v]_h
	=
	-\tfrac1h \big( [v]_h - v \big)
	\label{eq:ODE_mollification}
\end{equation}
with initial condition $[v]_h(0) = v_o$.
More precisely, let $X$ be a separable Banach space, $v_o \in X$ and $v \in L^r(0,T;X)$ for some $1 \leq r \leq \infty$.
Later on, we will mainly consider $X=L^q(\Omega,\R^N)$ and $X=W^{1,p}_{u_{\ast}}(\Omega,\R^N)$.
For $h \in (0,T]$, we define the mollification
\begin{equation}
	[v]_h(t) :=
	e^{-\frac{t}{h}} v_o + \tfrac1h \int_0^t e^\frac{s-t}{h} v(s) \,\ds
	\label{eq:time_mollification}
\end{equation}
for any $t \in [0,T]$.
The basic properties of this mollification procedure are collected in the following lemma.
For the proofs of the statements, we refer to \cite[Appendix~B]{BDM13} and \cite{Naumann}.
\begin{lemma}\label{lem:mollification:estimate_and_continuous_convergence}
Let $X$ be a separable Banach space and $v_o \in X$.
If $v \in L^r(0,T;X)$ for some $1 \leq r \leq \infty$, then the mollification $[v]_h$ defined according to \eqref{eq:time_mollification} satisfies $[v]_h \in L^r(0,T;X)$ and for any $t_o \in (0,T]$ there holds the bound
\begin{align*}
	\big\| [v]_h \big\|_{L^r(0,t_o;X)}
	\leq
	\| v \|_{L^r(0,t_o;X)}
	+ \Big[ \tfrac{h}{r} \Big( 1 - e^{-\frac{t_o r}{h}} \Big) \Big]^\frac{1}{r} \| v_o\|_X,
\end{align*}
where the bracket $[\ldots]^\frac{1}{r}$ is interpreted as $1$ in the case $r=\infty$.
Furthermore, in the case $r<\infty$ we have that $[v]_h \to v$ in $L^r(0,T;X)$ in the limit $h \downarrow 0$.
If $v \in C^0([0,T];X)$ and $v_o = v(0)$, then $[v]_h \in C^0([0,T];X)$ with $[v]_h(0) = v_o$ and moreover, we have that $[v]_h \to v$ in $C^0([0,T];X)$ as $h \downarrow 0$.
\end{lemma}

Next, we recall the results from \cite[Lemma 2.3]{BDM14-6}, which yields the convergence $f(x,[v]_h, D[v]_h) \to f(x,v,Dv)$ as $h \downarrow 0$, provided that \eqref{eq:integrand} holds.
Note that the property \eqref{ineq:mollification_convexity_estimate} is evident from the proof of \cite[Lemma 2.3]{BDM14-6}.
\begin{lemma}\label{lem:mollification_convexity}
    Let $f \colon \Omega \times \R^N \times \R^{Nn} \to [0,\infty)$ satisfy ~\eqref{eq:integrand} and $T >0$. Assume that $v \in L^{1}(0,T;W^{1,1}(\Omega, \mathds{R}^{N}))$, with $f(\cdot, v, Dv) \in L^{1}(\Omega_{T})$, and $v_o \in W^{1,1}(\Omega, \mathds{R}^{N})$, with $f(\cdot, v_o, Dv_o) \in L^{1}(\Omega_{T})$. Then, it follows that $f(\cdot, [v]_h, D[v]_h) \in L^{1}(\Omega_{T})$ and 
    \begin{align}\label{ineq:mollification_convexity_estimate}
        f(x, [v]_h, D[v]_h) \le [f(x, v, Dv)]_h.
    \end{align}
    Furthermore, we have that
    \begin{align*}
        \lim_{h\downarrow 0} \int_0^{T} \int_{\Omega} f(x, [v]_h, D[v]_h) \,\dx\dt = \int_0^{T} \int_{\Omega} f(x,v,Dv) \,\dx\dt.
    \end{align*}
\end{lemma}

\subsection{The initial condition}
Here, we show that any variational solution defined as in Definition~\ref{Definition:variational_solution} admits the initial condition in the $L^{q+1}$-sense.
We emphasize that this also holds for general domains $E$, i.e., domains that may shrink in time.
In the proof of the following result, for $t\in[0,T)$ and $\sigma >0$ we denote by
\begin{align}\label{inner_parallel_set}
    E^{t,\sigma} := \big\{ x \in E^{t} : \dist (x, \partial E^{t} ) > \sigma \big\}
\end{align}
the inner parallel set of $E^{t}$.
\begin{lemma}\label{lem:initial_condition_convergence}
Assume that $E$ satisfies \eqref{eq:boundary_E_zero_measure}, that the variational integrand $f\colon \Omega \times \mathds{R}^{N} \times \mathds{R}^{Nn} \to [0,\infty)$ satisfies \eqref{eq:integrand}, and that the initial and the lateral boundary values $u_o$ and $u_{\ast}$ fulfill \eqref{compatibility_condition_for_u_o_and_u_ast}.
Then, any variational solution to \eqref{eq:system} in the sense of Definition~\ref{Definition:variational_solution} admits the initial values $u(0) = u_o$ in the $L_{\mathrm{loc}}^{q+1}$-sense, i.e.,
    \begin{align*}
        \lim_{h \to 0} \tfrac{1}{h} \int_o^{h} \Vert u(t)-u_o \Vert_{L^{q+1}(K, \mathds{R}^{N})}^{q+1} \,\dt =0
    \end{align*}
    holds for any compact set $K \subset E^{0}$.
\end{lemma}

\begin{proof}
    Let $\epsilon >0$ and suppose that $E^{0, \epsilon}$ is the inner parallel set of $E^{0}$ defined by~\eqref{inner_parallel_set}.
    Since $E$ is relatively open and $\overline{E^{0,\epsilon}}$ is compact, there exists $t_{\epsilon} > 0$ such that
    \begin{align}\label{inner_parallel_set_E_0,epsilon}
        E^{0,\epsilon} \times [0, t_{\epsilon}) \subset E. 
    \end{align}
    Consider a standard mollifier $\phi \in C_{0}^{\infty}(B_{1}(0), \mathds{R}_{\ge 0})$.
    Setting $\phi_{\epsilon}(x) := \epsilon^{-n} \phi ( \tfrac{x}{\epsilon})$, such that $\phi_{\epsilon} \in C_{0}^{\infty}(B_{\epsilon}(0), \mathds{R}_{\ge 0})$, we define the mollification of the initial datum $u_o$ by
    \begin{align}\label{initial_value:u_0_epsilon}
        u_o^{(\epsilon)} := u_{\ast} + ((u_o - u_{\ast})\chi_{E^{0, 2 \epsilon}}) \ast \phi_{\epsilon}.
    \end{align}
    Note that $u_o^{(\epsilon)} - u_{\ast} \in
    C^{\infty}(E^{0,\epsilon}, \mathds{R}^{N})$ with
    $\spt \big(u_o^{(\epsilon)} - u_{\ast} \big) \subset
    \overline{E^{0,\epsilon}}$ and $u_o^{(\epsilon)} \to u_o$ in
    $L^{q+1}(E^0, \mathds{R}^{N})$ as $\epsilon \downarrow 0$.
    Further, by \eqref{compatibility_condition_for_u_o_and_u_ast} it follows that
    \begin{align*}
        0 \le \int_{\Omega} f \big(x, u_o^{(\epsilon)}, Du_o^{(\epsilon)} \big) \,\dx < \infty 
    \end{align*}
    holds for any $\epsilon > 0$. This enables us to consider the time independent extension of the initial value $u_o^{(\epsilon)}$ to $\Omega_{\tau}$, where $\tau \in (0,t_\eps)$, as a comparison map in the variational inequality ~\eqref{eq:variational_inequality}. Therefore,  
    we deduce that
    \begin{align}\label{ineq:first_inequality_developed_in_initial_condition_proof}
        \int_{E^\tau}	\b \big[u(\tau),u_o^{(\epsilon)} \big] \,\dx 
	&\leq
	\iint_{E\cap \Omega_\tau} f \big( x,u_o^{(\epsilon)},Du_o^{(\epsilon)} \big) \,\dx\dt + \int_{E^0} \b \big[u_o,u_o^{(\epsilon)} \big] \,\dx  \nonumber\\
    &= \tau \int_{\Omega} f \big( x,u_o^{(\epsilon)},Du_o^{(\epsilon)} \big) \,\dx + \int_{E^0} \b \big[ u_o,u_o^{(\epsilon)} \big] \,\dx
    \end{align}
    for a.e.~$\tau \in (0,t_\eps)$.
    Further, by the definition of $\b[\cdot, \cdot]$ we obtain the inequality
    \begin{align}\label{identity:inequality_for_b[]}
        \b[u(\tau),&u_{o}]
        \leq
        \b \big[u(\tau),u_{o}^{(\epsilon)} \big]
        + \tfrac{1}{q+1} \Big( |u_{o}|^{q+1} - \big|u_{o}^{(\epsilon)} \big|^{q+1} \Big)
        + |u(\tau)|^q \big| u_{o}^{(\epsilon)}  - u_o \big|.
    \end{align}
    For every compact set $K\subset E^0$, we can choose $\eps>0$ so small
    that $K\subset E^{0,\eps}$. In view
    of~\eqref{inner_parallel_set_E_0,epsilon}, this implies $K\subset
    E^\tau$ for every $\tau\in(0,t_\eps)$. 
    Therefore, ~\eqref{identity:inequality_for_b[]} combined with \eqref{ineq:first_inequality_developed_in_initial_condition_proof} and Hölder's inequality gives us that
    \begin{align*}
        \int_{K} &\b[u(\tau),u_o] \,\dx
        \\ &\leq
        \int_{K} \b \big[u(\tau),u_{o}^{(\epsilon)} \big] \,\dx 
        + \tfrac{1}{q+1} \int_{K} \Big( |u_{o}|^{q+1} - \big|u_{o}^{(\epsilon)} \big|^{q+1} \Big) \,\dx \\
        & \phantom{=}
        + \bigg( \int_{K} |u(\tau)|^{q+1} \,\dx \bigg)^{\frac{q}{q+1}} 
        \bigg( \int_{K} \vert u_{o}^{(\epsilon)} - u_{o} \vert^{q+1} \,\dx \bigg)^{\frac{1}{q+1}}
         \\&\le
        \tau  \int_{\Omega} f \big(x,u_{o}^{(\epsilon)},Du_{o}^{(\epsilon)} \big) \,\dx
        + \int_{E^0} \b \big[u_{o},u_{o}^{(\epsilon)} \big] \,\dx \\
        &\phantom{=}
        + \tfrac{1}{q+1} \int_{K} \Big( |u_{o}|^{q+1} - \big|u_{o}^{(\epsilon)} \big|^{q+1} \Big) \,\dx \\
        &\phantom{=}
        +\Vert u \Vert_{L^{\infty}(0,T;L^{q+1}(\Omega, \mathds{R}^{N}))}^q
        \bigg( \int_{K} \vert u_{o}^{(\epsilon)} - u_{o} \vert^{q+1} \,\dx \bigg)^{\frac{1}{q+1}}.
    \end{align*} 
    Now, integrating the preceding inequality over $\tau \in [0,h]$ for some $0< h < t_{\epsilon}$ and dividing both sides by $h$ yields
    \begin{align*}
        \tfrac{1}{h} \int_0^{h}& \int_{K}	\b[u(\tau) , u_o] \,\dx \d\tau \\
        &\le
        \tfrac{h}{2} \int_{\Omega} f\big(x,u_o^{(\epsilon)},Du_o^{(\epsilon)}\big) \,\dx
	  +\int_{E^{0}} \b\big[u_o , u_o^{(\epsilon)}\big] \,\dx \\
        &\phantom{=}
        + \tfrac{1}{q+1} \int_{K} \big( |u_{o}|^{q+1} - \big|u_{o}^{(\epsilon)} \big|^{q+1} \big) \,\dx \\
        &\phantom{=}
        +\Vert u \Vert_{L^{\infty}(0,T;L^{q+1}(\Omega, \mathds{R}^{N}))}^q
        \bigg( \int_{K} \vert u_{o}^{(\epsilon)} - u_{o} \vert^{q+1} \,\dx \bigg)^{\frac{1}{q+1}}.
    \end{align*}
    By passing to the limit $h \downarrow 0$, we obtain that
    \begin{align*}
         \limsup_{h \downarrow 0} \tfrac{1}{h} &\int_0^{h} \int_{K}	\b[u(\tau) , u_o] \,\dx \d\tau \\
        &\le
	  \int_{E^{0}} \b\big[u_o , u_o^{(\epsilon)}\big] \,\dx 
        + \tfrac{1}{q+1} \int_{K} \Big( |u_o|^{q+1} - \big|u_o^{(\epsilon)} \big|^{q+1} \Big) \,\dx \\
        &\phantom{=}
        +\Vert u \Vert_{L^{\infty}(0,T;L^{q+1}(\Omega, \mathds{R}^{N}))}^q
        \bigg( \int_{K} \vert u_{o}^{(\epsilon)} - u_{o} \vert^{q+1} \,\dx \bigg)^{\frac{1}{q+1}}.
    \end{align*}
    Since $u_o^{(\epsilon)} \to u_o$ in $L^{q+1}(E^0,\mathds{R}^{N})$ as $\epsilon \downarrow 0$, we derive that 
    \begin{align}\label{limit_in_initial_condition_proof}
           \limsup_{h \downarrow 0} \tfrac{1}{h} \int_0^{h} \int_{K}	\b[u(\tau) , u_o] \,\dx \d\tau =0.
     \end{align}
    By Lemma~\ref{lem:technical_lemma_2} and Lemma~\ref{lem:technical_lemma} we obtain the estimate
    \begin{align*}
        \vert u(\tau) - u_o \vert^{q+1}
        &\le
        c(q) \Big(\vert u(\tau) \vert^{\frac{q+1}{2}} + \vert u_o \vert^{\frac{q+1}{2}} \Big) \Big\vert u(\tau)^{\frac{q+1}{2}} - {u_o}^{\frac{q+1}{2}} \Big\vert \nonumber \\
        &\le
        c(q) \Big(\vert u(\tau) \vert^{\frac{q+1}{2}} + \vert u_o \vert^{\frac{q+1}{2}} \Big) \b[u(\tau), u_o]^{\frac{1}{2}}.
    \end{align*}
    Hence, by integrating over the space-time cylinder and applying Hölder's inequality, we find that
        \begin{align*}
            \tfrac{1}{h} &\int_0^{h} \Vert u(\tau)-u_o \Vert_{L^{q+1}(K, \mathds{R}^{N})}^{q+1} \,\d\tau
            \\ &\le
            \tfrac{c(q)}{h} \bigg( \int_0^{h} \int_{K} \vert u(\tau) \vert^{q+1} + \vert u_o \vert^{q+1} \,\dx\d\tau\bigg)^{\frac{1}{2}}
            \bigg( \int_0^{h} \int_{K} \b[u(\tau), u_o] \,\dx\d\tau \bigg)^{\frac{1}{2}}.
        \end{align*}
    The claim follows by taking the limit $h \downarrow 0$ and using ~\eqref{limit_in_initial_condition_proof}.
\end{proof}

\begin{remark}
If $u_o$ and $u_{\ast}$ coincide in $\Omega \setminus E^0$, then we obtain the convergence in the $L^{q+1}(\Omega, \mathds{R}^{N})$-sense, i.e.,
\begin{align*}
    \lim_{h \downarrow 0} \tfrac{1}{h} \int_0^{h} \Vert u(t)-u_o \Vert_{L^{q+1}(\Omega, \mathds{R}^{N})}^{q+1} \,\dt =0.
\end{align*}
\end{remark}

\subsection{Continuity in time}
\label{sec:continuity in time}
In this section we establish continuity in time for any variational solution to \eqref{eq:system} in the sense of Definition \ref{Definition:variational_solution} in nondecreasing domains.
\begin{proposition}
\label{prop:continuity-nondecreasing}
    Suppose that $E \subset \Omega_{T}$ is a relatively open noncylindrical domain satisfying \eqref{eq:boundary_E_zero_measure} and \eqref{nondecreasing_condition}, and let $u$ be a variational solution to \eqref{eq:system} in the sense of Definition~\ref{Definition:variational_solution}. Then, we have that $u \in C^{0}([0,T]; L^{q+1}(\Omega, \mathds{R}^{N}))$.
\end{proposition}
\begin{proof}
As in the proof of Lemma \ref{lem:initial_condition_convergence}, we consider the standard mollification of the initial value $u_o$ defined by
\begin{align}\label{continuity:u_0_epsilon}
    u_o^{(\epsilon)} := u_{\ast} + \big((u_o - u_{\ast})\chi_{E^{0, 2 \epsilon}} \big) \ast \phi_{\epsilon}.
\end{align}
Since \eqref{nondecreasing_condition} holds, it follows that $\spt \big(u_o^{(\epsilon)} \big) \subset E^{t}$ for all $t \in [0,T)$.
Now, let $[u]_{\lambda, \epsilon}$ with $\lambda>0$ be the time mollification of $u$ according to \eqref{eq:time_mollification} with initial values $u_o^{(\epsilon)}$.
Using $[u]_{\lambda, \epsilon}$ as a comparison map in the variational inequality ~\eqref{eq:variational_inequality}, we obtain that
\begin{align*}
    \int_{E^\tau} & \b \big[u(\tau),[u]_{\lambda, \epsilon}(\tau)\big] \,\dx + \iint_{E\cap \Omega_\tau} f(x,u,Du) \,\dx\dt \nonumber \\
	&\leq
	\iint_{E\cap \Omega_\tau} f \big( x,[u]_{\lambda, \epsilon},D[u]_{\lambda, \epsilon} \big) \,\dx\dt 
	+\iint_{E\cap \Omega_\tau} \partial_t [u]_{\lambda, \epsilon} \cdot \big( \power{[u]_{\lambda, \epsilon}}{q} - \power{u}{q} \big) \,\dx\dt \nonumber  \nonumber \\
	&\phantom{=}
	+\int_{E^0} \b \big[ u_o,u_o^{(\epsilon)} \big] \,\dx \nonumber \\
    &\leq \iint_{E\cap \Omega_\tau} \big[ f(x,u,Du) \big]_{\lambda, \epsilon} \,\dx\dt 
	- \tfrac{1}{h_\ell}\iint_{E\cap \Omega_\tau} \big( [u]_{\lambda, \epsilon}-u \big) \cdot \big( \power{[u]_{\lambda, \epsilon}}{q} - \power{u}{q} \big)  \,\dx\dt \nonumber \\
	&\phantom{=}
	+\int_{E^0} \b \big[ u_o,u_o^{(\epsilon)} \big] \,\dx \nonumber \\
    &\leq \iint_{E\cap \Omega_\tau} \big[ f(x,u,Du) \big]_{\lambda, \epsilon} \,\dx\dt + \int_{E^0} \b \big[ u_o,u_o^{(\epsilon)} \big] \,\dx
\end{align*}
for a.e.~$\tau \in [0,T)$, where $\big[ f(x,u,Du) \big]_{\lambda, \epsilon}$ is defined according to \eqref{eq:time_mollification} with initial values $f \big( x,u_o^{(\epsilon)},Du_o^{(\epsilon)} \big)$.
In the penultimate estimate, we applied Lemma~\ref{lem:mollification_convexity} to the convex function $f$, and in the last line we discarded the second term, which is negative by Lemma~\ref{lem:technical_lemma}. 
Reorganizing the terms, we find that
\begin{align}
    \int_{E^\tau} \b \big[ u(\tau),[u]_{\lambda, \epsilon}(\tau) \big] \,\dx &\leq \iint_{E\cap \Omega_\tau} \big[ \big[ f(x,u,Du) \big]_{\lambda, \epsilon} - f(x,u,Du) \big]\,\dx\dt
    \label{ineq:integral_estimate_for_b(u,u_lambda_epsilon)} \\
    &\phantom{=}
    + \int_{E^0} \b \big[ u_o,u_o^{(\epsilon)} \big] \,\dx.
    \nonumber
\end{align}
Since $E$ is nondecreasing, the integrand of the first term on the right-hand side of \eqref{ineq:integral_estimate_for_b(u,u_lambda_epsilon)} vanishes in $\Omega_{T} \setminus E$.
Thus, by \eqref{eq:ODE_mollification} and since $f$ is non-negative, we have that
\begin{align}
    \iint_{E \cap \Omega_{\tau}} &\big[ \big[ f(x,u,Du) \big]_{\lambda, \epsilon} -f(x,u,Du) \big] \,\dx\dt
    \nonumber \\
    &\qquad= \iint_{\Omega_{\tau}} \big[ \big[ f(x,u,Du) \big]_{\lambda, \epsilon} - f(x,u,Du) \big] \,\dx\dt
    \nonumber \\
    &\qquad= -\lambda \iint_{\Omega_{\tau}} \partial_t \big[ f(x,u,Du) \big]_{\lambda, \epsilon} \,\dx\dt
    \nonumber \\
    &\qquad=\lambda \int_{\Omega} f \big(x, u_o^{(\epsilon)}, Du_o^{(\epsilon)} \big) \,\dx - \lambda \int_{\Omega\times\{\tau\}} f \big( x, [u]_{\lambda, \epsilon}, D[u]_{\lambda, \epsilon} \big) \,\dx
    \nonumber \\
    &\qquad\le \lambda \int_{\Omega} f \big( x, u_o^{(\epsilon)}, Du_o^{(\epsilon)} \big) \,\dx.
    \label{eq:aux-continuity-time}
\end{align}
Let $(\epsilon_{i})_{i \in \mathds{N}}$ be a sequence such that $\epsilon_{i} \downarrow 0$ as $i \to \infty$. We define
\begin{align*}
    \lambda_{i} :=
    \epsilon_{i} \bigg( 1 + \int_{\Omega}  f \big( x,u_o^{(\epsilon_{i})} ,Du_o^{(\epsilon_{i})} \big) \,\dx \bigg)^{-1}
    >0.
\end{align*}
Choosing $\epsilon_i$ and $\lambda_i$ in \eqref{ineq:integral_estimate_for_b(u,u_lambda_epsilon)} and estimating the right-hand side by \eqref{eq:aux-continuity-time}, we conclude that
\begin{align*}
    \int_{E^\tau} \b \big[ u(\tau),[u]_{\lambda_i, \epsilon_i}(\tau) \big] \,\dx
    &\leq
    \epsilon_{i} + \int_{E^0} \b \big[ u_o,u_o^{(\epsilon_i)} \big] \,\dx.
\end{align*}
Now, using Lemma~\ref{lem:technical_lemma_2}, Lemma \ref{lem:technical_lemma}, Hölder's inequality, and Lemma \ref{lem:mollification:estimate_and_continuous_convergence} we derive that
\begin{align*}
    \sup_{\tau \in [0, T)} &\int_{E^\tau} \big\vert [u]_{\lambda_i, \epsilon_i}(\tau) -u(\tau) \big\vert^{q+1} \,\dx
    \\ &\leq
    c(q) \sup_{\tau \in [0, T)} \bigg( \int_{E^{\tau}}	\vert u(\tau) \vert^{q+1} + \big\vert [u]_{\lambda_i, \epsilon_i}(\tau) \big\vert^{q+1} \,\dx\bigg)^{\frac{1}{2}}
    \\ &\phantom{=}
    \cdot \bigg( \int_{E^\tau} \b \big[ u(\tau), [u]_{\lambda_i, \epsilon_i}(\tau) \big] \,\dx \bigg)^{\frac{1}{2}}
    \\ &\leq
    c(q) \Big( \|u\|_{L^\infty(0,T;L^{q+1}(\Omega,\R^N))} + \|u_o\|_{L^{q+1}(\Omega,\R^N)} + \|u_\ast\|_{L^{q+1}(\Omega,\R^N)} \Big)^\frac{q+1}{2}
    \\ &\phantom{=}
    \cdot \bigg( \epsilon_i + \int_{E^0} \b \big[ u_o,u_o^{(\epsilon_i)} \big] \,\dx  \bigg)^\frac{1}{2}.
\end{align*}
Since the right-hand side vanishes as $i \to \infty$, and since we have that $[u]_{\lambda_{i}, \epsilon_i} \in C^{0}([0,T]; L^{q+1}(\Omega, \mathds{R}^{N}))$ for all $i \in \mathds{N}$, this implies the claim of the lemma. 
\end{proof}

\subsection{Density of smooth functions}
In this section, we provide conditions under which $C_{0}^{\infty}(E, \mathds{R}^{N})$ is dense in $V_{q}^{p,0}(E)$.
For the sake of generality, we consider domains $E$ that are allowed to shrink in time.
Moreover, we replace the measure density condition~\eqref{ineq:lower_measure_bound} by the weaker condition that for every $t\in(0,T)$
\begin{equation}
  \label{p-fat}
  \R^n\setminus E^t \mbox{ is uniformly $p$-fat with a parameter $\alpha>0$} 
\end{equation}
in the sense of Definition~\ref{def:p-fat}; see Remark~\ref{rem:V=V} for the relation between \eqref{ineq:lower_measure_bound} and \eqref{p-fat}.
To control the speed at which $E$ may shrink, with the complementary excess
\begin{align*}
    \boldsymbol{e}^{c}(E^{s}, E^{t}):= \sup_{x \in \Omega \setminus E^{t}} \dist(x, \Omega \setminus E^{s}),
\end{align*}
we impose the one-sided growth condition that there exists a modulus of continuity, i.e., a continuous function $\omega \colon [0,\infty)\rightarrow[0,\infty)$
with $\omega(0)=0$, such that
  \begin{align}\label{one sided growth condition}
   \boldsymbol{e}^{c}(E^{s},E^{t})\leq\omega(t-s)\quad\text{for }0\leq s\leq t<T.
  \end{align}
Note that in the special case of nondecreasing domains, \eqref{one sided growth condition} is always satisfied. 
In the following, we generalize the corresponding statements for the case $q=1$ in \cite[Section 5.3.1]{BDSS18}.
First, we define the space
\begin{align*}
    V_\mathrm{cpt}^{p,q}(E):=
    \Big\{ u \in V_{q}^{p,0}(E) : \exists\, \sigma > 0 \text{ with }\spt(u(t)) \subset \overline{E^{t,\sigma}} \text{ for a.e.~$t\in[0,T)$} \Big\},
\end{align*}
and prove an intermediate result.
\begin{lemma}\label{intermediate_density_argument}
    Suppose that \eqref{one sided growth condition} holds with a modulus of continuity $\omega$.
    Then  $C_{0}^{\infty}(E, \mathds{R}^{N})$ is dense in $V_\mathrm{cpt}^{p,q}(E)$ with respect to the norm $\Vert \cdot \Vert_{V_{q}^{p}(E)}$.
    Furthermore, if $u \in V_\mathrm{cpt}^{p,q}(E)$, then there exist $\sigma > 0$ and an approximating sequence $\phi_{k} \in C_{0}^{\infty}(E, \mathds{R}^{N})$ with $\phi_{k} \to u$ in the norm topology of $V_{q}^{p,0}(E)$ such that $\spt(\phi_{k}(t)) \subset \overline{E^{t,\sigma}}$ for all $k \in \mathds{N}$ and $t \in [0,T)$.
\end{lemma}
\begin{proof}
    Fix $u \in V_\mathrm{cpt}^{p,q}(E)$. Then, there exists $\sigma > 0$ with $\spt(u(t))\subset \overline{E^{t,2\sigma}}$ for a.e.~$t \in [0,T)$.
    Without loss of generality, we may assume that $u \equiv 0$ in the set $\Omega \times ([0,2\delta] \cup [T-2\delta,T])$ for some $\delta >0$, since the space of such maps is dense in $V_\mathrm{cpt}^{p,q}(E)$ with respect to the norm $\Vert \cdot \Vert_{V_{q}^{p}(E)}$. Consider a backward-in-time mollifying kernel $\zeta \in C_{0}^{\infty}(B_{1}(0) \times (0,1), [0, \infty))$ with $\iint_{\R^n \times (0,\infty)} \zeta \,\dx\dt =1$, and define $\zeta_{\epsilon}:= \epsilon^{-n-1} \zeta \big(\tfrac{x}{\epsilon}, \tfrac{t}{\epsilon}\big)$ for $0<\epsilon < \min\big\{ \delta, \tfrac{\sigma}{2},\omega^{-1}(\frac{\sigma}{2}) \big\}$, and
    \begin{align*}
        u_{\epsilon}:=u \ast \zeta_{\epsilon}.
    \end{align*}
    It follows that $u_{\epsilon} \in C^{\infty}(\Omega_{T},\mathds{R}^{N})$, and that $u_{\epsilon} \to u$ in
    $L^{p} \big(0,T;W_{0}^{1,p} (\Omega, \mathds{R}^{N}) \big) \cap L^{q+1}(\Omega_{T}, \R^N)$ as $\eps\downarrow0$.
    We now want to show that
    \begin{align}\label{formula for u_epsilon}
        \spt(u_{\epsilon}(t)) \subset \overline{E^{t, \sigma}} \quad \text{for a.e.~$t \in [0,T]$}.
    \end{align}
    Since we have assumed that $u \equiv 0$ in $\Omega \times ([0,2\delta] \cup [T-2\delta,T])$, and since $\varepsilon < \delta$, \eqref{formula for u_epsilon} follows directly from the construction for $t \in [0,2\delta] \cup [T-\delta,T]$.
    It remains to consider $t \in [2\delta, T-\delta]$ and $x \in \Omega \setminus \overline{E^{t,\sigma}}$.
    Note that
    \begin{align*}
        u_{\epsilon}(x,t) = \iint_{B_{\epsilon}(x) \times (t-\epsilon,t)} u(y,s)\zeta_{\epsilon}(x-y,t-s) \,\dy\ds,
    \end{align*}
    and fix $s \in (t-\epsilon,t)$ and $y \in B_{\epsilon}(x)$.
    Since $u \equiv 0$ on $(B_{\epsilon}(x) \times (t-\epsilon,t)) \setminus E$, we only need to examine the case $y \in B_{\epsilon}(x) \cap E^{s}$. 
    Further, let $\widetilde{x} \in \Omega \setminus E^t$ such that $\vert x- \widetilde{x} \vert = \dist(x, \Omega \setminus E^t) \leq \sigma$.
    By \eqref{one sided growth condition} and our choice of $\epsilon$, we deduce the estimate
    \begin{align*}
        \dist(y, \partial E^s ) &= \dist(y, \Omega \setminus E^s ) \\
        &\le \dist(\widetilde{x}, \Omega \setminus E^s ) + \vert y - \widetilde{x} \vert \\
        &\le \power{e}{c}(E^s, E^t) + \vert y-x \vert + \vert x- \widetilde{x} \vert \\
        &\le \omega(t-s) +\epsilon + \sigma \\
        &\le \omega \big( \omega^{-1}(\tfrac{\sigma}{2}) \big) + \tfrac{\sigma}{2} + \sigma \\
        &<2\sigma.
    \end{align*}
    It follows that $y \in E^s \setminus \overline{E^{s,2\sigma}}$, and in turn that $u$ vanishes a.e.~in $B_{\epsilon}(x) \times (t-\epsilon,t)$.
    Therefore, we have that $u_\varepsilon \equiv 0$ on $\Omega \setminus \overline{E^{t,\sigma}}$ for any $t \in [2\delta, T-\delta]$, i.e., we have proven \eqref{formula for u_epsilon}.
    Since $E$ is relatively open in $\Omega_T$, in particular we conclude that $u_\epsilon \equiv 0$ on $\partial E$.
    By a standard cut-off argument, we finally obtain that every $u_\epsilon$ can be approximated by maps from the space $C_{0}^{\infty}(E, \mathds{R}^{N})$.
    This implies the density of $C_{0}^{\infty}(E, \mathds{R}^{N})$ in $V_\mathrm{cpt}^{p,q}(E)$, as well as the second statement of the lemma.
\end{proof}

Next, analogous to \eqref{eq:curly-Vp0}, we consider the subspace of $V^{p,0}_q(E)$ given by
\begin{equation*}\label{def:VV}
    \mathcal{V}^{p,0}_q(E):=
    \Big\{u\in V^{p,0}_q(E)\colon u(t)\in W^{1,p}_0(E^t,\R^N) \mbox{ for a.e. }t\in[0,T)\Big\}.
\end{equation*}
Further, let $\widetilde{\eta} \in C^{0,1}(\mathds{R})$ be a cut-off function such that $\widetilde{\eta} \equiv 0$ in $(-\infty,1]$, $\widetilde{\eta}(r):=r-1$ for $r \in (1,2)$ and $\widetilde{\eta} \equiv 1$ in $[2, \infty)$.
For $\sigma >0$, we define
\begin{align}\label{definition_eta_sigma}
    \eta_{\sigma}(x,t):=
    \widetilde{\eta}\bigg( \frac{\dist(x,\Omega \setminus E^t)}{\sigma} \bigg)
    \quad \text{for }(x,t) \in \Omega_{T}.
\end{align}
With this notation at hand, we deduce the following convergence result.
\begin{lemma}\label{weak_convergence_curoff_density_part}
Assume that $E$ satisfies~\eqref{p-fat}, let $\eta_\sigma$ be given by \eqref{definition_eta_sigma}, and let $u \in \mathcal{V}^{p,0}_q(E)$.
Then, we have that
\begin{align}
    \eta_{\sigma} u \wto u \text{ weakly in } \mathcal{V}_{q}^{p,0}(E)\text{ as }\sigma \downarrow 0.
\end{align}
\end{lemma}
\begin{proof}
First, by the dominated convergence theorem, we have that
\begin{equation}
    \mbox{$\eta_{\sigma}u \to u$ strongly in $L^{\max \{p,q+1\}}(E, \mathds{R}^{N})$ as $\sigma \downarrow 0$.}
    \label{eq:aux-density}
\end{equation}
Since $\vert D\eta_{\sigma} \vert \le \tfrac{1}{\sigma}$ and by Hardy's inequality \eqref{ineq:Hardy_inequality}, we obtain that
    \begin{align*}
        \iint_{\Omega_{T}} \vert D(\eta_{\sigma}u)\vert^{p} \,\dx\dt
        &\le
        2^{p-1} \iint_{E} \vert Du \vert^{p} \,\dx\dt + 2^{p-1} \int_0^{T} \int_{E^{t,\sigma} \setminus E^{t,2 \sigma}} \vert D \eta_{\sigma} \vert^{p} \vert u \vert^{p} \,\dx\dt \\
        &\le 2^{p-1} \iint_{E} \vert Du \vert^{p} \,\dx\dt + 2^{2p-1} \int_0^{T} \int_{E^t} \bigg| \frac{u(x,t)}{\dist(x, \partial E^t)} \bigg|^{p} \,\dx\dt \\
        &\le c(n,p,\alpha) \iint_{E} \vert Du \vert^{p} \,\dx\dt.
    \end{align*}
    This implies that $(\eta_\sigma u)_{\sigma>0}$ is bounded in $\mathcal{V}^{p,0}_q(E)$.
    Thus, $(\eta_\sigma u)_{\sigma>0}$ admits a limit with respect to weak convergence in $\mathcal{V}^{p,0}_q(E)$, which is determined by \eqref{eq:aux-density}.
\end{proof}

Now we are in the position to prove the following density result under the $p$-fatness condition~\eqref{p-fat}.
\begin{proposition}\label{prop:C_0_inf_is_dense_in_curly_V_q_p,0}
    Assume that $E$ satisfies~\eqref{p-fat} and~\eqref{one sided
      growth condition}. Then $C_{0}^{\infty}(E,\mathds{R}^{N})$ is
    dense in $\mathcal{V}^{p,0}_q(E)\subset V^{p,0}_q(E)$ with respect to the norm topology in $V^{p,0}_q(E)$.
\end{proposition}
  
\begin{proof}
    Choose a function $u \in \mathcal{V}^{p,0}_q(E)$ and define $u_{\sigma}:=
    \eta_{\sigma} u$, where $\sigma > 0$ and $\eta_{\sigma}$ is as in
    \eqref{definition_eta_sigma}. By definition of $\eta_\sigma$, we find that $u_{\sigma} \in V_\mathrm{cpt}^{p,q}(E)$, and by Lemma~\ref{weak_convergence_curoff_density_part} we obtain the weak convergence $u_{\sigma} \wto u$ in $\mathcal{V}^{p,0}_q(E)$ as $\sigma \downarrow 0$. 
    By Mazur's lemma, for any $k \in \N$ there exists a convex combination
    \begin{align*}
        V_\mathrm{cpt}^{p,q}(E) \ni
        v_{k}
        :=\sum_{\ell = 1}^{N_{k}} \lambda_{\ell} u_{\sigma_{\ell}},
    \end{align*}
    where $N_{k} \in \mathds{N}$, $\lambda_{\ell} \in [0,1]$ and $\sigma_{\ell} > 0$ for $\ell \in \{1,\ldots,N_{k}\}$, such that $v_{k} \to u$ strongly in $\mathcal{V}_{q}^{p,0}(E)$ as $k \to \infty$.
    This gives us the density of $V_\mathrm{cpt}^{p,q}(E)$ in $\mathcal{V}_{q}^{p,0}(E)$.
    Combining this with Lemma~\ref{intermediate_density_argument} implies the statement of the proposition. 
\end{proof}

Since $V^{p,0}_q(E) = \mathcal{V}^{p,0}_q(E)$ provided that the stronger measure density condition \eqref{ineq:lower_measure_bound} holds, see Remark \ref{rem:V=V}, we immediately conclude the following result.
\begin{corollary}\label{cor:C_0_inf_is_dense_in_V_q_p,0}
Assume that $E$ satisfies~\eqref{ineq:lower_measure_bound} and~\eqref{one sided growth condition}. Then $C_{0}^{\infty}(E,\mathds{R}^{N})$ is dense in $V^{p,0}_q(E)$ with respect to the norm topology in $V^{p,0}_q(E)$.
\end{corollary}

\begin{remark}\label{rem:C_0_inf_is_dense_in_V_p,0}
Applying Corollary \ref{cor:C_0_inf_is_dense_in_V_q_p,0} with the choice $q=p-1$, we deduce that $C_{0}^{\infty}(E,\mathds{R}^{N})$ is dense in $V^{p,0}(E)$ with respect to the norm topology in $V^{p,0}(E)$.
\end{remark}

\section{Proof of Theorem~\ref{thm:existence_in_nondecreasing_domains}}
\label{sec:existence}
The proof of Theorem~\ref{thm:existence_in_nondecreasing_domains} given in this section is divided into several steps.

\subsection{Time discretization}
For some $\ell \in \N$, fix a step size $h_\ell := \frac{T}{\ell}$ and set $u_{\ell,0} := u_o \in L^{q+1}(\Omega,\R^N)$, where $u_o$ has been extended from $E^0$ to $\Omega$ by $u_{\ast}$.
Next, for $i=1,\ldots,\ell$ iteratively define $u_{\ell,i}$ as the minimizer of the elliptic variational functional
\begin{align*}
	\F_\ell[w;u_{\ell,i-1}]
	:=
	\int_\Omega f(x,w,Dw) \,\dx
	+ \tfrac{1}{h_\ell} \int_\Omega \b[u_{\ell,i-1},w] \,\dx
\end{align*}
in the class $w \in L^{q+1}(\Omega, \mathds{R}^{N}) \cap V_{ih_\ell}$; see Proposition \ref{prop:direct-method} below for the existence of such a minimizer.
In this context, by \eqref{compatibility_condition_for_u_o_and_u_ast}$_{3}$ with $\phi \equiv 0$ we have that
\begin{align*}
    0
    \leq
	\F_\ell[u_{\ast};u_{\ell,i-1}]
	&=
	\int_\Omega f(x,u_{\ast},Du_{\ast}) \,\dx + \tfrac{1}{h_\ell} \int_\Omega \b[u_{\ell,i-1},u_{\ast}] \,\dx < \infty.
\end{align*}
This implies that $0 \leq \F_\ell[u_{\ell,i};u_{\ell,i-1}] < \infty$, and in particular, that
\begin{equation}
	\int_\Omega f(x,u_{\ell,i},Du_{\ell,i}) \,\dx < \infty
	\label{eq:finite_integral_uli}
\end{equation}
for any $i = 1,\ldots,\ell$.

\begin{proposition}
\label{prop:direct-method}
Let $\overline{u} \in L^{q+1}(\Omega, \mathds{R}^{N})$, $h >0$ and $t \in [0,T]$.
Then there exists a minimizer $u \in L^{q+1}(\Omega, \mathds{R}^{N}) \cap V_t$ of the elliptic variational functional
$$
	\F_h[w;\overline{u}]
	:=
	\int_\Omega f(x,w,Dw) \,\dx
	+ \tfrac{1}{h} \int_\Omega \b[\overline{u},w] \,\dx
$$
in the class of functions $w \in L^{q+1}(\Omega, \mathds{R}^{N}) \cap V_t$. 
\end{proposition}
\begin{proof}
    Take a minimizing sequence $(u_{j})_{j \in \mathds{N}} \subset L^{q+1}(\Omega, \mathds{R}^{N}) \cap V_t$ such that 
    \begin{align*}
        \lim_{j \to \infty} \F_h[u_{j};\overline{u}] = \inf_{w \in L^{q+1}(\Omega, \mathds{R}^{N}) \cap V_t} \F_h[w;\overline{u}].
    \end{align*}
    Since the infimum is finite by \eqref{compatibility_condition_for_u_o_and_u_ast}$_{3}$ with $\phi \equiv 0$ and the admissibility of $u_\ast$, without loss of generality, we may assume that $\F_h[u_{j};\overline{u}] < \infty$ for any $j \in \N$.
    In particular, since $(\F_h[u_{j};\overline{u}])_{j \in \N}$ is convergent, this implies that the sequence is bounded.
    Using Lemma \ref{lem:b_and_abs_u_estimate} to bound $\b[\overline u,u_j]$ from below, we obtain that 
    \begin{align*}
        \tfrac{1}{q+1} \vert u_{j} \vert^{q+1} \le 2 \b[\overline{u},u_{j}] + 2^{2+\frac{1}{q}} \tfrac{q}{q+1}\vert \overline{u} \vert^{q+1}.
    \end{align*}
    Together with the coercivity condition \eqref{eq:integrand}$_3$, this implies that
    \begin{align*}
        \tfrac{1}{2h(q+1)} \int_\Omega &\vert u_{ j} \vert^{q+1}  \,\dx + \nu \int_{\Omega} |Du_{j}|^p \,\dx \\
	&\leq
	\int_{\Omega} f(x,u_{j},Du_{j}) \,\dx +  \tfrac{1}{h} \int_\Omega  \b[\overline{u},u_{j}] \,\dx + 2^{1+\frac{1}{q}}\tfrac{q}{h(q+1)} \int_{\Omega}  \vert \overline{u} \vert^{q+1} \,\dx  \\
    &\leq
	\sup_{j \in \N} \F_h[u_{j};\overline{u}] + 2^{1+\frac{1}{q}}\tfrac{q}{h(q+1)} \int_{\Omega}  \vert \overline{u} \vert^{q+1} \,\dx.
    \end{align*}
    As a consequence, the sequence $(u_{j})_{j \in \mathds{N}}$ is bounded in $L^{q+1}(\Omega, \mathds{R}^{N}) \cap W^{1,p}(\Omega, \mathds{R}^{N})$.
    Therefore, there exists a subsequence (still denoted by $u_{j}$) and a limit map $u \in L^{q+1}(\Omega, \mathds{R}^{N}) \cap W^{1,p}(\Omega, \mathds{R}^{N})$ such that
    \begin{equation*}
    	\left\{
    	\begin{array}{ll}
    		u_{j} \wto u
    		&\text{weakly in $W^{1,p}(\Omega, \mathds{R}^{N})$ as $j \to \infty$,} \\[5pt]
    		u_{j} \wto u
    		&\text{weakly in $L^{q+1}(\Omega, \mathds{R}^{N})$ as $j \to \infty$.}
    	\end{array}
    	\right.
    \end{equation*}
    Since $u_{j} \in V_t$ for all $j \in \mathds{N}$, we have that $u \in V_t$.
    Next, we observe that $\F_h$ is lower semicontinuous with respect to strong convergence in $L^{q+1}(\Omega, \mathds{R}^{N}) \cap W^{1,p}(\Omega, \mathds{R}^{N})$. 
    To this end, consider a strongly convergent sequence $w_j\to w$ in $L^{q+1}(\Omega, \mathds{R}^{N}) \cap W^{1,p}(\Omega, \mathds{R}^{N})$. 
    By passing to a subsequence if necessary, we can assume $(w_j,Dw_j)\to (w,Dw)$ a.e.~in $\Omega$, as $j\to\infty$.
    Therefore, Fatou's lemma implies
    \begin{align*}
        \F_h[w;\overline{u}] &= \int_\Omega f(x,w,Dw) \,\dx
	   + \tfrac{1}{h} \int_\Omega \b[\overline{u},w] \,\dx \\
        &= \int_\Omega \liminf_{j \to \infty} f(x,w_{j},Dw_{j}) \,\dx
	   + \tfrac{1}{h} \int_\Omega \liminf_{j \to \infty} \b[\overline{u},w_{j}] \,\dx \\
        &\leq \liminf_{j \to \infty} \F_h[w_{j};\overline{u}].
    \end{align*}
    Since $\F_h[\cdot,\overline{u}]$ is convex, by \cite[Corollary 3.9]{Brezis} $\F_h$ is lower semicontinuous also with respect to weak convergence in $L^{q+1}(\Omega, \mathds{R}^{N}) \cap W^{1,p}(\Omega, \mathds{R}^{N})$. This gives us that
    \begin{align*}
        \F_h[u;\overline{u}] \le \liminf_{j \to \infty} \F_h[u_{j};\overline{u}] =\inf_{w \in L^{q+1}(\Omega) \cap V_t} \F_h[w;\overline{u}].
    \end{align*}
    Therefore, $u \in L^{q+1}(\Omega, \mathds{R}^{N}) \cap V_t$ is the desired minimizer.
\end{proof}

\subsection{Reformulation of the minimizing property}
Note that for any index $i \in \{1,\ldots,\ell\}$, $s \in (0,1)$ and $w \in V_{i h_\ell} \cap L^{q+1}(\Omega,\R^N)$, the convex combination
\begin{align*}
	w_s :=
	(1-s) u_{\ell,i} + sw
\end{align*}
is an admissible competitor for $u_{\ell,i}$.
Thus, the minimality of $u_{\ell,i}$ in this class gives us that
\begin{align*}
	\F_\ell[u_{\ell,i};u_{\ell,i-1}]
	\leq
	\F_\ell[w_s;u_{\ell,i-1}].
\end{align*}
Together with the convexity assumption \eqref{eq:integrand}$_2$, this implies that
\begin{align*}
	\int_\Omega f(x,u_{\ell,i},Du_{\ell,i}) \,\dx 
	&\leq
	\int_\Omega (1-s) f(x,u_{\ell,i},Du_{\ell,i}) + s f(x,w,Dw) \,\dx \\
	&\phantom{=}
	+ \tfrac{1}{h_\ell} \int_\Omega \b[u_{\ell,i-1}, w_s] - \b[u_{\ell,i-1}, u_{\ell,i}] \,\dx \\
	&=
	\int_\Omega (1-s) f(x,u_{\ell,i},Du_{\ell,i}) + s f(x,w,Dw) \,\dx \\
	&\phantom{=}
	+ \tfrac{1}{h_\ell} \int_\Omega \tfrac{1}{q+1} \big( |w_s|^{q+1} - |u_{\ell,i}|^{q+1} \big) - s \power{u_{\ell,i-1}}{q} \cdot (w - u_{\ell,i}) \,\dx.
\end{align*}
Reabsorbing the first term in the penultimate line of the preceding inequality (which is allowed by \eqref{eq:finite_integral_uli}) and dividing by $s>0$, we find that
\begin{align*}
	\int_\Omega f(x,u_{\ell,i},Du_{\ell,i}) \,\dx 
	&\leq
	\int_\Omega f(x,w,Dw) \,\dx \\
	&\phantom{=}
	+ \tfrac{1}{h_\ell} \int_\Omega \tfrac{1}{(q+1)s} \big( |w_s|^{q+1} - |u_{\ell,i}|^{q+1} \big) - \power{u_{\ell,i-1}}{q} \cdot (w - u_{\ell,i}) \,\dx.
\end{align*}
Since $\frac{1}{q+1} |\cdot|^{q+1}$ is convex, the map $s \mapsto \tfrac{1}{(q+1)s} \big( |w_s|^{q+1} - |u_{\ell,i}|^{q+1} \big)$ is monotone and converges a.e.~in $\Omega$ to $\power{u_{\ell,i}}{q} \cdot (w - u_{\ell,i})$ in the limit $s \downarrow 0$.
Therefore, passing to the limit in the preceding inequality by means of the dominated convergence theorem shows that
\begin{align}
	\int_\Omega &f(x,u_{\ell,i},Du_{\ell,i}) \,\dx 
	\nonumber \\ &\leq
	\int_\Omega f(x,w,Dw) \,\dx
	+ \tfrac{1}{h_\ell} \int_\Omega \big( \power{u_{\ell,i}}{q} - \power{u_{\ell,i-1}}{q} \big) \cdot (w - u_{\ell,i}) \,\dx
    \label{eq:mm_crucial_inequality}
\end{align}
holds true for any $w \in V_{i h_\ell} \cap L^{q+1}(\Omega,\R^N)$, $i=1,\ldots,\ell$.

\subsection{Energy estimates and weak convergence}
Let $t_{\ell,i} := i h_\ell$ for $i=-1,\ldots,\ell$ and glue the minimizers $u_{\ell,i}$ together to a function $u_\ell \colon \Omega \times (-h_\ell,T] \to \R^N$ that is piecewise constant with respect to the time variable.
More precisely, we set
\begin{align*}
	u_\ell(t) := u_{\ell,i}
	\quad \text{for $t \in J_{\ell,i} := (t_{\ell,i-1}, t_{\ell,i}]$ with $i = 0, \ldots, \ell$.}
\end{align*}
In order to deduce suitable energy estimates for $u_\ell$, we choose the admissible function $w \equiv u_{\ast}$ in \eqref{eq:mm_crucial_inequality}.
First, note that
\begin{align}\label{ineq_second_integrand_on_the_right-hand_side_1}
	\big| \power{u_{\ell,i-1}}{q} \cdot u_{\ell,i} \big|
	\leq
	|u_{\ell,i-1}|^q |u_{\ell,i}|
	\leq
	\tfrac{q}{q+1} |u_{\ell,i-1}|^{q+1} + \tfrac{1}{q+1} |u_{\ell,i}|^{q+1}
\end{align}
holds by Cauchy--Schwarz and Young's inequalities.
Multiplying \eqref{eq:mm_crucial_inequality} with $w \equiv u_{\ast}$ by $h_\ell$ and using ~\eqref{ineq_second_integrand_on_the_right-hand_side_1}, we find that
\begin{align*}
	&\iint_{\Omega \times J_{\ell,i}} f(x,u_{\ell,i}, Du_{\ell,i}) \,\dx\dt
	+ \tfrac{q}{q+1} \int_\Omega |u_{\ell,i}|^{q+1} \,\dx \\
	&\leq
	h_\ell \int_\Omega f(x,u_{\ast}, Du_{\ast}) \,\dx
	+ \tfrac{q}{q+1} \int_\Omega |u_{\ell,i-1}|^{q+1} \,\dx
    + \int_\Omega \big( \power{u_{\ell,i}}{q} - \power{u_{\ell,i-1}}{q} \big) \cdot u_{\ast} \,\dx
\end{align*}
holds true for any $i=1,\ldots,\ell$.
Summing up the preceding inequalities from $i=1,\ldots,m$ for some $m \leq \ell$,  reabsorbing the sum of the second terms on the right-hand side into the left-hand side and using ~\eqref{compatibility_condition_for_u_o_and_u_ast}$_{3}$ with $\phi \equiv 0$, we arrive at
\begin{align*}
	&\iint_{\Omega \times (0,mh_\ell]} f(x,u_\ell,Du_\ell) \,\dx\dt
	+ \tfrac{q}{q+1} \int_\Omega |u_\ell(m h_\ell)|^{q+1} \,\dx \\
	&\leq
	m h_\ell \int_\Omega f(x,u_{\ast}, Du_{\ast}) \,\dx
	+ \tfrac{q}{q+1} \int_\Omega |u_o|^{q+1} \,\dx
    + \int_\Omega \big( \power{u_{\ell}(mh_\ell)}{q} - \power{u_o}{q} \big) \cdot u_{\ast} \,\dx \\
    &\leq
	m h_\ell \int_\Omega f(x,u_{\ast}, Du_{\ast}) \,\dx
	+ \tfrac{q}{q+1} \int_\Omega |u_o|^{q+1} \,\dx
    \\ &\phantom{=}
    + \int_\Omega \big( \vert u_{\ell}(mh_\ell) \vert^{q} \vert u_{\ast} \vert + \vert u_o \vert^{q} \vert u_{\ast} \vert \big)  \,\dx.
\end{align*}
By Young's inequality, we find that
\begin{align*}
    \vert u_{\ell}(mh_\ell) \vert^{q} \vert u_{\ast} \vert + \vert u_o \vert^{q} \vert u_{\ast} \vert 
    \le \tfrac{q}{2(q+1)} |u_{\ell}(mh_\ell)|^{q+1} + \tfrac{q}{q+1} |u_o|^{q+1} + \tfrac{2^{q} +1}{q+1} |u_{\ast}|^{q+1}.
\end{align*}
Using this in the last integral on the right-hand side of the penultimate inequality, and reabsorbing the term involving $u_{\ell}(mh_\ell)$ into the left-hand side yields 
\begin{align*}
	&\iint_{\Omega \times (0,mh_\ell]} f(x,u_\ell,Du_\ell) \,\dx\dt
	+ \tfrac{q}{2(q+1)} \int_\Omega |u_\ell(m h_\ell)|^{q+1} \,\dx \\
	&\leq
	m h_\ell \int_\Omega f(x,u_{\ast}, Du_{\ast}) \,\dx
	+ \tfrac{2q}{q+1} \int_\Omega |u_o|^{q+1} \,\dx
    +  \tfrac{2^{q} +1}{q+1} \int_\Omega \vert u_{\ast} \vert^{q+1} \,\dx.
\end{align*}
Note that both terms on the left-hand side are non-negative.
On the one hand, omitting the first term, taking the supremum over $m=1,\ldots,\ell$ and using that $m h_\ell \leq T$, we infer
\begin{align}
	\sup_{t \in (0,T)} &\int_\Omega |u_\ell(t)|^{q+1} \,\dx
	\nonumber \\ &\leq
	c(q) \bigg( T \int_\Omega f(x,u_{\ast}, Du_{\ast}) \,\dx + \int_\Omega  \vert u_{\ast} \vert^{q+1} \,\dx \bigg)
    +2  \int_\Omega |u_o|^{q+1} \,\dx.
	\label{eq:energy_est_1_mm}
\end{align}
On the other hand, omitting the second term on the left-hand side of the penultimate inequality, choosing $m=\ell$ and using the coercivity assumption \eqref{eq:integrand}$_3$, we conclude that
\begin{align}
	\nu \iint_{\Omega_T} |Du_\ell|^p \,\dx\dt
	&\leq
	\iint_{\Omega_T} f(x,u_\ell,Du_\ell) \,\dx\dt \nonumber \\
	&\leq
	 T \int_\Omega f(x,u_{\ast}, Du_{\ast}) \,\dx + c(q) \int_\Omega |u_o|^{q+1} + \vert u_{\ast} \vert^{q+1} \,\dx. 
	\label{eq:energy_est_2_mm}
\end{align}
The energy estimates \eqref{eq:energy_est_1_mm} and \eqref{eq:energy_est_2_mm} show that the sequence $(u_\ell)_{\ell \in \N}$ is bounded in the spaces $L^\infty(0,T;L^{q+1}(\Omega,\R^N))$ and $ L^p(0,T;W^{1,p}_{u_{\ast}}(\Omega,\R^N))$.
Therefore, there exists a subsequence $\mathfrak{K} \subset \N$ and a limit map
\begin{align*}
	u \in L^\infty(0,T;L^{q+1}(\Omega,\R^N)) \cap  L^p \big(0,T;W^{1,p}_{u_{\ast}}(\Omega,\R^N) \big)
\end{align*}
such that
\begin{equation}
	\left\{
	\begin{array}{ll}
		u_\ell \wsto u
		&\text{weakly$^\ast$ in $L^\infty(0,T;L^{q+1}(\Omega,\R^N))$ as $\mathfrak{K} \ni \ell \to \infty$,} \\[5pt]
		u_\ell \wto u
		&\text{weakly in $L^p\big( 0,T;W^{1,p}_{u_{\ast}}(\Omega,\R^N) \big)$ as $\mathfrak{K} \ni \ell \to \infty$.}
	\end{array}
	\right.
	\label{eq:weak_convergence_mm}
\end{equation}

\subsection{Boundary values}\label{Section:Boundary_values}
In this section, we show that $u \in V_{q}^{p}(E)$, i.e. that the limit map admits the correct boundary values. This means that $u(t)=u_{\ast}$ a.e.~in $\Omega \setminus E^{t}$. We proceed as in \cite[Section 4.1.4]{BDSS18}. Let $\ell \in \mathds{N}$ and define
\begin{align*}
    E^{(\ell)} := \bigcup\limits_{i=1}^{\ell} Q_{\ell, i},
\end{align*}
with
\begin{align*}
    Q_{\ell, i}:= E^{t_{\ell, i}} \times I_{\ell, i} \text{ and } I_{\ell, i}:= [ t_{\ell, i-1}, t_{\ell, i} ).
\end{align*}
By \eqref{nondecreasing_condition}, we have that $E \subset E^{(\ell)}$, and by construction it follows that $u_{\ell} \equiv u_{\ast}$ a.e.~in $\Omega_{T} \setminus E^{(\ell)}$.
We fix a point $z_{o}:=(x_{o}, t_o) \in \Omega_{T} \setminus \overline{E}$, and introduce the notation
\begin{align*}
\Lambda_{\delta}(t_o):= (t_o - \delta, t_o + \delta) \text{ and } Q_{\delta}(z_{o}):= B_{\delta}(x_{o}) \times \Lambda_{\delta}(t_o). 
\end{align*}
At this stage, we claim that
\begin{align}\label{cylinder_claim}
    \exists \, \epsilon >0, \ell_{0} \in \mathds{N} \text{ such that } \forall \, \ell \ge \ell_{0}: Q_{\epsilon}(z_o) \subset \Omega_{T} \setminus \overline{E^{(\ell)}}.
\end{align}
Since $\Omega_{T} \setminus \overline{E}$ is open, we find $\epsilon > 0$ such that $Q_{2\epsilon}(z_o) \subset \Omega_{T} \setminus \overline{E}$.
By~\eqref{nondecreasing_condition} it follows that $B_{2\epsilon}(x_o) \subset \Omega \setminus \overline{E^{t}}$ for all $t \in (0, t_o+2\epsilon)$.
Now, assume that $\ell > \frac{T}{\epsilon}$ holds, which implies $h_{\ell} < \epsilon$.
Then there exists $i_o \in \{ 1, \ldots, \ell\}$ such that $i_o h_{\ell} \in [ t_o+\epsilon, t_o+2\epsilon)$.
As a consequence, we have that $B_{2\epsilon}(x_o) \subset \Omega \setminus \overline{E^{t_{\ell, i}}}$ for any $\N \ni i \le i_o$. 
Taking into consideration how $E^{(\ell)}$ was constructed, we find that
\begin{align*}
    Q_{\epsilon}(z_o) \subset B_{2\epsilon}(x_o) \times \Lambda_{\epsilon}(t_o) \subset \Omega_{T} \setminus \overline{E^{(\ell)}}.
\end{align*}
Therefore, \eqref{cylinder_claim} holds for any $\ell_{0} \in \mathds{N}$ satisfying $\ell_{0} > \frac{T}{\epsilon}$.

Next, by \eqref{eq:weak_convergence_mm}$_{2}$ and \eqref{compatibility_condition_for_u_o_and_u_ast}$_{2}$, we have that $u_{\ell}-u_{\ast} \wto u-u_{\ast}$ weakly in $L^{p}(Q_{\epsilon}(z_o), \mathds{R}^{N})$ as $\mathfrak{K} \ni \ell \to \infty$.
Thus, by lower semicontinuity we deduce that
\begin{align*}
    \iint_{Q_{\epsilon}(z_o)} \vert u - u_{\ast} \vert^{p} \,\dx\dt \le \liminf_{\mathfrak{K} \ni \ell \to \infty } \iint_{Q_{\epsilon}(z_o)} \vert u_{\ell} - u_{\ast} \vert^{p} \,\dx\dt = 0.
\end{align*}
Therefore, $u \equiv u_{\ast}$ a.e.~in $Q_{\epsilon}(z_o)$.
Since the point $z_o\in \Omega_{T} \setminus \overline{E}$ was arbitrary, this gives us that $u \equiv u_{\ast}$ a.e.~in $\Omega_{T} \setminus \overline{E}$.
Furthermore, the assertion that $u \equiv u_{\ast}$ a.e.~in $\Omega \times \{t\} \setminus \overline{E}$ for a.e.~$t \in [0,T)$ follows from a Fubini-type argument. Indeed, using assumption \eqref{eq:boundary_E_zero_measure} we observe that
\begin{align*}
    0=\L^{n+1}(\partial E) \equiv \L^{n+1} \big( \overline{E} \setminus E \big) = \int_0^{T} \L^{n} \big( (\Omega \times \{ t \}) \cap \big(\overline{E} \setminus E \big) \big) \, \dt,
\end{align*}
and therefore, $\L^{n} \big( (\Omega \times \{ t \}) \cap \big(\overline{E} \setminus E \big) \big)=0$ for a.e.~$t \in [0,T)$.
Finally, combining this with the identity
\begin{align*}
    \big(\Omega \setminus E^{t} \big) \times \{t\} = \big( (\Omega \times \{t\}) \setminus \overline{E} \big) \cup \big( (\Omega \times \{ t \}) \cap \big(\overline{E} \setminus E \big) \big)
\end{align*}
implies that $u(t)=u_{\ast}$ a.e.~in $\Omega \setminus E^{t}$ for a.e.~$t \in [0,T)$.
Consequently, we have that $u(t) \in V_{t}$ for the corresponding times $t$, and thus altogether that $u \in V_{q}^{p}(E)$.

\subsection{Convergence almost everywhere}\label{Section:Convergence_almost_everywhere_non_decreasing}
In this section, we analyze the convergence of $u_\ell$ to $u$ further by  proceeding as in \cite[Section 6.3]{BDMS18-1}.
To this end, note that $w = u_{\ell,i-1}$ is an admissible comparison map in \eqref{eq:mm_crucial_inequality}, since $u_{\ell,i-1} \in V_{i h_\ell}$ for $i=2, \ldots, \ell$ by \eqref{nondecreasing_condition}.
Together with Lemma \ref{lem:technical_lemma}, this gives us that
\begin{align*}
	\int_\Omega &f(x,u_{\ell,i},Du_{\ell,i}) \,\dx
	+ \tfrac{c(q)}{h_\ell} \int_\Omega \Big| \power{u_{\ell,i}}{\frac{q+1}{2}} - \power{u_{\ell,i-1}}{\frac{q+1}{2}} \Big|^2 \,\dx \\
	&\leq
	\int_\Omega f(x,u_{\ell,i},Du_{\ell,i}) \,\dx
	+ \tfrac{1}{h_\ell} \int_\Omega \big( \power{u_{\ell,i}}{q} - \power{u_{\ell,i-1}}{q} \big) \cdot (u_{\ell,i} - u_{\ell,i-1}) \,\dx \\
	&\leq
	\int_\Omega f(x,u_{\ell,i-1},Du_{\ell,i-1}) \,\dx.
\end{align*}
Setting $2 \leq i_1 < i_2 \leq \ell$ and iterating the preceding inequality from $i=i_1+1, \ldots, i_2$, we obtain that
\begin{align*}
	\int_\Omega &f(x,u_{\ell,i_2},Du_{\ell,i_2}) \,\dx
	+ \tfrac{c(q)}{h_\ell} \sum_{i=i_1+1}^{i_2} \int_\Omega \Big| \power{u_{\ell,i}}{\frac{q+1}{2}} - \power{u_{\ell,i-1}}{\frac{q+1}{2}} \Big|^2 \,\dx \\
	&\leq
	\int_\Omega f(x,u_{\ell,i_1},Du_{\ell,i_1}) \,\dx \\
	&\leq
	\int_\Omega f(x,u_{\ell,i_o},Du_{\ell,i_o}) \,\dx
\end{align*}
for any $i_o \in \{1, \ldots, i_1\}$.
In the last line, we have used that the second term on the left-hand side of the penultimate inequality is non-negative and can thus be omitted when iterating further.
Denoting the backward difference quotient by
\begin{align*}
	\Delta_{-h_\ell} w
	:=
	\tfrac{1}{h_\ell} (w(t) - w(t-h_\ell)),
\end{align*}
we rewrite the preceding inequality as
\begin{align*}
	\int_\Omega &f(x,u_\ell(i_2 h_\ell),Du_\ell(i_2 h_\ell)) \,\dx
	+ c(q) \int_{i_1 h_\ell}^{i_2 h_\ell} \int_\Omega \Big| \Delta_{-h_\ell} \power{u_\ell}{\frac{q+1}{2}} \Big|^2 \,\dx \\
	&\leq
	\int_\Omega f(x,u_\ell(t),Du_\ell(t)) \,\dx
\end{align*}
for $t \in [h_\ell, i_1 h_\ell]$.
Let $0 < \epsilon < \tau < T$, $\ell > \frac{4 T}{\epsilon}$ and choose $i_1 := \big\lfloor \frac{\epsilon}{h_\ell} \big\rfloor$ and $i_2 := \big\lceil \frac{\tau}{h_\ell} \big\rceil$.
Since the left-hand side of the preceding inequality is independent of $t$, by taking the average integral over $t \in [h_\ell, \epsilon - h_\ell]$ and using the energy estimate \eqref{eq:energy_est_2_mm}, we conclude that
\begin{align}
	\int_\Omega &f(x,u_\ell(\tau),Du_\ell(\tau)) \,\dx
	+ c(q) \iint_{\Omega \times (\epsilon,\tau)} \Big| \Delta_{-h_\ell} \power{u_\ell}{\frac{q+1}{2}} \Big|^2 \,\dx
    \nonumber \\ &\leq
	\bint_{h_\ell}^{\epsilon - h_\ell} \int_\Omega f(x,u_\ell,Du_\ell) \,\dx\dt
    \nonumber \\ &\leq
	\tfrac{1}{\epsilon - 2h_\ell} \bigg( 
	T \int_\Omega f(x,u_{\ast}, Du_{\ast}) \,\dx
	+ c(q)  \int_\Omega |u_o|^{q+1} + \vert u_{\ast} \vert^{q+1} \,\dx \bigg).
    \label{bound-time-derivative}
\end{align}
By the coercivity condition \eqref{eq:integrand}$_3$, this implies in particular that
\begin{align}
	\sup_{\tau \in (\epsilon,T)} &\int_\Omega |Du_\ell|^p \,\dx
	\nonumber \\ &\leq
	\tfrac{1}{(\epsilon - 2h_\ell)\nu} \bigg( T \int_\Omega f(x,u_{\ast}, Du_{\ast}) \,\dx
	+ c(q) \int_\Omega |u_o|^{q+1} + \vert u_{\ast} \vert^{q+1} \,\dx \bigg).
    \label{Linfty-Lp-bound}
\end{align}
Because of~\eqref{eq:energy_est_1_mm}, \eqref{bound-time-derivative},
and~\eqref{Linfty-Lp-bound}, the assumptions of Lemma \ref{lem:compactness} are satisfied in $\Omega \times (\epsilon,T)$.
Therefore, there exists a subsequence $\mathfrak{K} \subset \N$ such that
\begin{equation}
	\power{u_\ell}{\frac{q+1}{2}} \to \power{u}{\frac{q+1}{2}}
	\quad \text{strongly in $L^1 \big(\Omega \times (\epsilon,T),\R^N \big)$ as $\mathfrak{K} \ni \ell \to \infty$}.
    \label{eq:strong_U_l_powerq+1_half_convergence_mm}
\end{equation}
Since $\epsilon>0$ was arbitrary, we find that there exists another subsequence still denoted by $\mathfrak{K} \subset \N$, such that
\begin{equation}
	u_\ell \to u
	\quad \text{a.e.~in $\Omega_T$ as $\mathfrak{K} \ni \ell \to \infty$.}
	\label{eq:ae_convergence_mm}
\end{equation}
Moreover, we will now have a closer look on the time derivative of our constructed solutions. Omitting the first term on the left-hand side in \eqref{bound-time-derivative} gives us an estimate for the finite difference quotient of $\power{u_\ell}{\frac{q+1}{2}}$. Therefore, there exists a subsequence, still denoted by $\mathfrak{K} \subset \N$, such that $\Delta_{-h_\ell} \power{u_\ell}{\frac{q+1}{2}} \wto w$ in $L^2(\Omega \times (\epsilon,T),\R^N)$ as $\mathfrak{K} \ni \ell \to \infty$.
Combining this with \eqref{eq:strong_U_l_powerq+1_half_convergence_mm} yields for all $\phi \in C_{0}^{\infty}(\Omega\times(\eps,T),\R^N)$ that
\begin{align*}
    \iint_{\Omega \times (\epsilon,T)} \power{u}{\frac{q+1}{2}} \cdot\partial_{t}\phi \,\dx\dt
    &= \lim_{\mathfrak{K} \ni \ell \to \infty} \iint_{\Omega \times (\epsilon,T)} \power{u_\ell}{\frac{q+1}{2}} \cdot \Delta_{h_\ell} \phi  \,\dx\dt \\
    &= - \lim_{\mathfrak{K} \ni l \to \infty} \iint_{\Omega \times (\epsilon,T)} \Delta_{-h_\ell} \power{u_\ell}{\frac{q+1}{2}} \cdot  \phi  \,\dx\dt \\
    &= -\iint_{\Omega \times (\epsilon,T)}  w \cdot  \phi  \,\dx\dt.
\end{align*}
From this we conclude that $w = \partial_{t} \power{u}{\frac{q+1}{2}} \in L^2(\Omega \times (\epsilon,T),\R^N)$.
Furthermore, since we have that $\partial_{t} \power{\vert u \vert }{\frac{q+1}{2}} \in L^2(\Omega \times (\epsilon,T),\R^N)$ by the chain rule, it follows that
\begin{align*}
    \partial_{t} \boldsymbol{u}^{q+1} = \partial_{t} \power{u}{\frac{q+1}{2}} \vert u \vert^{\frac{q+1}{2}} + \power{u}{\frac{q+1}{2}} \partial_{t} \vert u \vert^{\frac{q+1}{2}} \in L^1 \big(\Omega \times (\epsilon,T),\R^N \big).
\end{align*}
Therefore, for any fixed $\epsilon > 0$ and for all $t_o, t \in ( \epsilon, T ] $, we obtain that
\begin{align*}
    \lim_{t \to t_o} \big\Vert \boldsymbol{u}^{q+1} (t) - \boldsymbol{u}^{q+1} (t_o) \big\Vert_{L^1(\Omega,\R^N)}
    \le
    \lim_{t \to t_o} \iint_{\Omega \times (t_o, t)} \big\vert \partial_{t} \boldsymbol{u}^{q+1} (s) \big\vert \,\dx\ds = 0.
\end{align*}
Together with Lemma \ref{lem:strong_L^q+1_convergence}, this implies that $u(t) \to u(t_o)$ in $L^{q+1}(\Omega,\R^N)$ as $t \to t_o$.
Since $\epsilon > 0$ was arbitrary, we conclude that $u \in C^{0}(0,T; L^{q+1}(\Omega,\R^N))$.

\subsection{Variational inequality for the limit map}\label{Section:Variational inequality for the limit map_nondecreasing}
First, setting $m:= \big\lfloor \tfrac{\tau}{h_\ell} \big\rfloor$, by definition of $u_\ell$ we deduce from \eqref{eq:mm_crucial_inequality} that
\begin{align}\label{crucial_inequality_appr}
    \iint_{\Omega_{\tau}}& f(x,u_{\ell},Du_{\ell}) \,\dx\dt \nonumber \\
    &=\sum_{i=1}^{m} \iint_{\Omega \times (t_{\ell, i-1}, t_{\ell, i}]} f(x, u_{\ell,i}, Du_{\ell,i}) \,\dx\dt \nonumber \\
    &\phantom{\sum_{i=1}^{m}}
    + \iint_{\Omega \times (mh_\ell, \tau]} f(x, u_{\ell,m}, Du_{\ell,m}) \,\dx\dt \nonumber \\
    &\le
    \sum_{i=1}^{m} \iint_{\Omega \times (t_{\ell, i-1}, t_{\ell, i}]} f(x, w, Dw) + \tfrac{1}{h_\ell} \big( \power{u_{\ell,i}}{q} - \power{u_{\ell,i-1}}{q} \big) \cdot (w - u_{\ell,i} )\,\dx\dt \nonumber \\
    &\phantom{=}
    + \iint_{\Omega \times (mh_\ell, \tau]} f(x, w, Dw) + \tfrac{1}{h_\ell} \big( \power{u_{\ell,i}}{q} - \power{u_{\ell,i-1}}{q} \big) \cdot (w - u_{\ell,i} ) \,\dx\dt \nonumber \\
    &= \iint_{\Omega_{\tau}} f(x,w,Dw) + \Delta_{-h_\ell} \big( \power{u_\ell}{q} \big) \cdot (w-u_{\ell}) \,\dx\dt
\end{align}
holds for any $\tau \in (0,T]$ and any map $w \in L^{q+1}(\Omega_{\tau}, \mathds{R}^{N}) \cap L^p \big(0,\tau;W^{1,p}_{u_{\ast}}(\Omega,\R^N)\big)$ with $w(t) \in  V_{\lceil \frac{t}{h_\ell} \rceil h_\ell}$ for a.e.~$t \in [0,T]$.
In particular, since $E$ is nondecreasing, for any map $v \in V_q^{p}(E)$ with $\partial_{t}v \in L^{q+1}(\Omega_{T}, \mathds{R}^{N})$ we find that $v(t) \in V_{\lceil \frac{t}{h_\ell} \rceil h_\ell}$ for a.e.~$t \in [0,T]$.
Thus, extending $v$ to negative times by $v(0)$, using it as a comparison map in \eqref{crucial_inequality_appr}, and applying Lemma \ref{lem:finite_integration_by_parts_formula} we obtain that
\begin{align}
    \iint_{\Omega_{\tau}} &f(x, u_{\ell}, Du_{\ell}) \,\dx\dt
    \nonumber\\ &\le
    \iint_{\Omega_{\tau}} f(x,v,Dv) \,\dx\dt + \iint_{\Omega_{\tau}} \Delta_{h_\ell}v \cdot \big( \power{v}{q}-\power{u_\ell}{q} \big)  \,\dx\dt
    \label{crucial_ineq_appr_finite_integration} \\ &\phantom{=}
    -\tfrac{1}{h_{\ell}} \iint_{\Omega \times (\tau - h_{\ell}, \tau)} \b[u_{\ell}(t),v(t+h_{\ell})] \,\dx\dt
    \nonumber\\ &\phantom{=}
    + \tfrac{1}{h_{\ell}} \iint_{\Omega \times (- h_{\ell}, 0)} \b[u_{\ell}(t),v(t)] \,\dx\dt + \delta_{1}(h_{\ell};\tau) + \delta_{2}(h_{\ell}),
    \nonumber
\end{align}
where the error terms $\delta_{1}(h_{\ell};\tau)$ and $ \delta_{2}(h_{\ell})$ are defined by
\begin{align*}
    \delta_{1}(h_{\ell};\tau)&:= \tfrac{1}{h_{\ell}} \iint_{\Omega_{\tau}} \b[v(t),v(t+h_{\ell})] \,\dx\dt, \\
    \delta_{2}(h_{\ell})&:=  \iint_{\Omega \times (- h_{\ell}, 0)} \Delta_{h_\ell}v \cdot \big( \power{v}{q}(t+h_{\ell}) - \power{u_{\ell}}{q}(t) \big) \,\dx\dt.
\end{align*}
To deal with the first boundary term, for $0<\delta<T$ and $0 \le t_o \le T - \delta$ we integrate \eqref{crucial_ineq_appr_finite_integration} over $\tau \in [t_o, t_o + \delta]$ and then divide by $\delta$.
Using that $f$ and $\b$ are non-negative, and that $\b[u_{\ell}(t),v(t)] = \b[u_o,v(0)]$ for $t \in (- h_{\ell}, 0)$, we end up with
\begin{align}\label{integrated_crucial_ineq_appr_finite_integration}
    \iint_{\Omega_{t_o}} f(x, u_{\ell}, Du_{\ell}) \,\dx\dt
    &\le
    \iint_{\Omega_{t_o+\delta}} f(x,v,Dv) \,\dx\dt
    \\ &\phantom{=}
    + \bint_{t_o}^{t_o+\delta} \iint_{\Omega_{\tau}} \Delta_{h_\ell}v \cdot \big( \power{v}{q}-\power{u_\ell}{q} \big) \,\dx\dt \,\d\tau
    \nonumber\\ &\phantom{=}
    -\tfrac{1}{\delta} \int_{t_o}^{t_o+\delta-h_{\ell}}\int_{\Omega } \b[u_{\ell}(t),v(t+h_{\ell})] \,\dx\dt 
    \nonumber\\ &\phantom{=}
    + \int_{\Omega} \b[u_o,v(0)] \,\dx
    + \delta_{1}(h_{\ell};t_o+\delta) + \delta_{2}(h_{\ell}).
    \nonumber
\end{align}
First, since the term on the left-hand side of \eqref{integrated_crucial_ineq_appr_finite_integration} is lower semicontinuous with respect to weak convergence in $L^p \big( 0,T;W^{1,p}(\Omega,\R^N) \big)$ by \eqref{eq:integrand}, by \eqref{eq:weak_convergence_mm}$_2$ we conclude that
\begin{equation}
    \iint_{\Omega_{t_o}} f(x,u,Du) \,\dx\dt
    \leq
    \liminf_{\mathfrak{K} \ni \ell \to \infty} \iint_{\Omega_{t_o}} f(x,u_{\ell}, Du_{\ell}) \, \dx\dt.
    \label{limit_var_ineq_1}
\end{equation}
Next, note that $\Delta_{h_\ell}v \to \partial_{t}v$ in $L^{q+1}(\Omega_{T}, \mathds{R}^{N})$.
Moreover, since $\power{u_{\ell}}{q}$ is bounded in $L^{\frac{q+1}{q}}(\Omega_{T}, \mathds{R}^{N})$ by \eqref{eq:energy_est_1_mm}, and since $\power{u_{\ell}}{q} \to \power{u}{q}$ a.e.~in $\Omega_{T}$ as $\mathfrak{K} \ni \ell \to \infty$ by \eqref{eq:ae_convergence_mm}, there exists a (not relabeled) subsequence such that $\power{u_{\ell}}{q} \wto \power{u}{q}$ as $\mathfrak{K} \ni \ell \to \infty$.
This gives us that
\begin{align}\label{limit_var_ineq_2}
    \lim_{\mathfrak{K} \ni \ell \to \infty} \iint_{\Omega_{\tau}}  \Delta_{h_\ell}v \cdot \big( \power{v}{q}-\power{u_\ell}{q} \big) \,\dx\dt = \iint_{\Omega_{\tau}} \partial_{t} v \cdot \big( \power{v}{q}-\power{u}{q} \big) \,\dx\dt
\end{align}
for every $\tau\in[t_o,t_o+\delta]$.
Moreover, since $\partial_{t} v \in L^{q+1}(\Omega_{T}, \mathds{R}^{N})$, we have that $v \in C^{0}([0,T]; L^{q+1}(\Omega_{T}, \mathds{R}^{N}))$.
Thus, together with \eqref{eq:ae_convergence_mm}, we find that $u_\ell(t) \to u(t)$ and $v(t+h_{\ell}) \to v(t)$ a.e.~in $\Omega_{T}$ as $\mathfrak{K} \ni \ell \to \infty$.
Since $\b$ is non-negative, applying Fatou's lemma yields
\begin{align}
    \bint_{t_o}^{t_o+\delta} \int_{\Omega } \b[u(t),v(t)] \,\dx\dt
    \leq
    \liminf_{\mathfrak{K} \ni \ell \to \infty} \tfrac{1}{\delta} \int_{t_o}^{t_o+\delta-h_{\ell}}\int_{\Omega } \b[u_{\ell}(t),v(t+h_{\ell})] \,\dx\dt.
    \label{limit_var_ineq_3}
\end{align}
Finally, since $\partial_{t}v \in L^{q+1}(\Omega_{T}, \mathds{R}^{N})$, by Lemma \ref{lem:finite_integration_by_parts_formula} we have that
\begin{align}\label{limit_var_ineq_4}
    \lim_{\mathfrak{K} \ni \ell \to \infty} \delta_{1}(h_{\ell};t_o+\delta)
    = 0 =
    \lim_{\mathfrak{K} \ni \ell \to \infty} \delta_{2}(h_{\ell}).
\end{align}
Using \eqref{limit_var_ineq_1}--\eqref{limit_var_ineq_4} to pass to the limit $\mathfrak{K} \ni \ell \to \infty$ in \eqref{integrated_crucial_ineq_appr_finite_integration}, we derive that
\begin{align}
    \iint_{\Omega_{t_o}} &f(x, u, Du) \,\dx\dt
    \nonumber \\ &\le
    \iint_{\Omega_{t_o+\delta}} f(x,v,Dv) \,\dx\dt
    + \bint_{t_o}^{t_o+\delta} \iint_{\Omega_{\tau}} \partial_t v \cdot \big( \power{v}{q}-\power{u}{q} \big) \,\dx\dt \,\d\tau
    \label{var_ineq_before_delta_limit} \\ &\phantom{=}
    -\bint_{t_o}^{t_o+\delta}\int_{\Omega } \b[u(t),v(t)] \,\dx\dt \,\d\tau
    + \int_{\Omega} \b[u_o,v(0)] \,\dx
    \nonumber
\end{align}
for all $t_o \in [0,T-\delta]$ and $\delta \in (0,T]$.
Since the term on the left-hand side and the first two terms on the right-hand side of \eqref{var_ineq_before_delta_limit} depend continuously on time, and by applying Lebesgue's differentiation theorem, we pass to the limit $\delta \downarrow 0$ in \eqref{var_ineq_before_delta_limit}.
This results in
\begin{align*}
     \iint_{\Omega_{t_o}} f(x, u, Du) \,\dx\dt \le &  \iint_{\Omega_{t_o}} f(x,v,Dv) \,\dx\dt\d\tau +  \iint_{\Omega_{t_o}}  \partial_{t}v \cdot \big( \power{v}{q}-\power{u}{q} \big)  \,\dx\dt \nonumber\\
    &- \int_{\Omega } \b[u(t_o),v(t_o)] \,\dx  + \int_{\Omega } \b[u_o,v(0)] \,\dx
\end{align*}
for a.e.~$t_o \in [0,T]$, and any comparison map $v \in V_q^{p}(E)$ with $\partial_{t}v \in L^{q+1}(\Omega_{T}, \mathds{R}^{N})$.
Therefore, $u$ is the desired variational solution to \eqref{eq:system} in the sense of Definition \ref{Definition:variational_solution}.
Observing that $u \in C^0([0,T);L^{q+1}(\Omega,\R^N))$ by Proposition \ref{prop:continuity-nondecreasing}, this concludes the proof of Theorem~\ref{thm:existence_in_nondecreasing_domains}.

\section{Proof of Theorem~\ref{thm:time_derivative_nondecreasing}}
\label{Section:Time derivative in the dual space(nondecreasing)}
\subsection{Preparation}
We introduce a similar notion of parabolic minimizers as in~\cite{Wieser}.
\begin{definition}[Parabolic minimizers]\label{def:parabolic_minimizers}
    Assume that $E \subset \Omega \times [0,T)$ is a relatively open noncylindrical domain, and let $f \colon \Omega \times \mathds{R}^{N} \times \mathds{R}^{Nn} \to [0,\infty]$ be a Cara\-théodory function.
    We say that $u \in V^{p}(E)$ with $\power{u}{q}\in L^1(E,\R^N)$ 
    is a parabolic minimizer if and only if
    \begin{align*}
     \iint_{E} \power{u}{q} \cdot \partial_{t}\phi +f(x,u,Du) \,\dx\dt \le \iint_{E} f(x,u+\phi, Du+D\phi) \,\dx\dt 
    \end{align*}
    holds for all $\phi \in C_{0}^{\infty}(E, \mathds{R}^{N})$.
\end{definition}

First, we show that any variational solution to \eqref{eq:system} is a parabolic minimizer. 

\begin{lemma}\label{lem:minimality_condition_for_var_sol}
Assume that $E$ satisfies \eqref{nondecreasing_condition}, that the variational integrand $f \colon \Omega \times \mathds{R}^{N} \times \mathds{R}^{Nn} \to [0,\infty)$ fulfills \eqref{eq:integrand} and \eqref{adjusted standard-coercivity and p-growth condition}, and that $u_o \in L^{q+1}(E^{0},\mathds{R}^{N})$ and $u_{\ast} \in W^{1,p}(\Omega, \mathds{R}^{N}) \cap L^{q+1}(\Omega, \mathds{R}^{N})$.
Then, any variational solution in the sense of Definition~\ref{Definition:variational_solution} is a parabolic minimizer in the sense of Definition \ref{def:parabolic_minimizers}.
\end{lemma}

\begin{proof}
Let $u_o^{(\epsilon)} \in C^\infty(\Omega,\R^N) \cap W^{1,p}(\Omega,\R^N)$ be defined by~\eqref{initial_value:u_0_epsilon}, and note that
\begin{align*}
    0 \le \int_{E^0} f \big( x,u_o^{(\epsilon)}, Du_o^{(\epsilon)} \big) \,\dx\dt < \infty
\end{align*}
holds by \eqref{adjusted standard-coercivity and p-growth condition}.
Further, since $E$ is nondecreasing and $\spt \big( u_o^{(\epsilon)}  - u_{\ast} \big) \subset E^{0}$ by construction, it follows that $\spt \big( u_o^{(\epsilon)}  - u_{\ast} \big) \subset E^t$ for all $t \in [0,T]$.
Moreover, we define $[u]_h(t)$ according to \eqref{eq:time_mollification} with initial values $u_o^{(\epsilon)}$, and for a test function $\phi \in C_{0}^{\infty}(E, \mathds{R}^{N})$, let $[\phi]_h(t)$ be defined according to \eqref{eq:time_mollification} with zero initial values.
Due to \eqref{nondecreasing_condition}, for $v_{h}(t):= [u]_h(t) + s [\phi]_h(t)$ with $s>0$ we find that $v_h = u_{\ast}$ a.e.~in $\Omega_{T} \setminus E$.
Hence, $v_h$ is an admissible comparison map in the variational inequality \eqref{eq:variational_inequality}, which holds for any $\tau \in [0,T]$ and thus in particular for $\tau=T$, since $u \in C^{0}([0,T];L^{q+1}(\Omega,\R^N))$ by Proposition \ref{prop:continuity-nondecreasing}.
Omitting the non-negative boundary term on the left-hand side, we obtain that
\begin{align}
    \iint_{E} &f(x,u,Du) \,\dx\dt
    \nonumber \\ &\leq
	\iint_{E} f(x,v_{h},Dv_{h}) \,\dx\dt
	+\iint_{E} \partial_t v_{h} \cdot \big( \power{v_{h}}{q} - \power{u}{q} \big) \,\dx\dt
	+\int_{E^0} \b \big[ u_o,u_o^{(\epsilon)} \big] \,\dx.
    \label{time_derivative_variational_inequality}
\end{align}
By the local Lipschitz condition \eqref{ineq:Lipschitz_condition} and the convergence $v_h \to u+s\phi$ in $L^{p}(0,T;W^{1,p}(\Omega, \mathds{R}^{N}))$ as $h \downarrow 0$, for the first term on the right-hand side of \eqref{time_derivative_variational_inequality} we obtain that
\begin{align*}
    \lim_{h \downarrow 0} \iint_{E} f(x,v_h,Dv_h)\,\dx\dt = \iint_{E} f(x,u+s\phi, Du+sD\phi) \,\dx\dt. 
\end{align*}
By definition of $v_{h}$, for the second term on the right-hand side of \eqref{time_derivative_variational_inequality} we find that
\begin{align}
    \iint_{E} &\partial_t v_{h} \cdot \big( \power{v_{h}}{q} - \power{u}{q} \big) \,\dx\dt
     \nonumber \\ &=
    \iint_{\Omega_{T}} \partial_t v_{h} \cdot \power{v_{h}} {q}  \,\dx\dt
    - \iint_{\Omega_{T}} \partial_t [u]_{h} \cdot \power{u}{q} \,\dx\dt
    - s\iint_{E} \partial_t [\phi]_{h} \cdot \power{u}{q} \,\dx\dt
    \nonumber\\ &=:
    I_{1} + I_{2} + I_{3}.
    \label{eq:split_time_derivative_integral}
\end{align}
Let $\psi(\cdot):=\tfrac{1}{q+1}\vert \cdot \vert^{q+1}$.
Using $\psi^{\prime}(v_{h}) = \power{v_{h}}{q}$, we compute
\begin{align*}
    I_{1} = \iint_{\Omega_{T}} \partial_{t}[\psi(v_{h})] \,\dx\dt = \int_{\Omega \times \{ T \}} \psi([u]_h + s [\phi]_h) \,\dx - \int_{\Omega \times \{ 0 \}} \psi \big(u_o^{(\varepsilon)} \big) \,\dx.
\end{align*}
By the convexity of $\psi$ and \eqref{ineq:mollification_convexity_estimate}, we obtain that
\begin{align*}
    -\partial_t [u]_{h} \cdot \power{u}{q} &= \tfrac{1}{h}([u]_{h}-u)\psi^{\prime}(u) \le \tfrac{1}{h}(\psi([u]_{h})- \psi(u)) \\
    &\le \tfrac{1}{h}([\psi(u)]_{h} - \psi(u)) = -\partial_t [\psi(u)]_{h},
\end{align*}
where the mollification of $\psi(u)$ is defined according to \eqref{eq:time_mollification} with initial values $\psi \big( u_o^{(\epsilon)} \big)$.
Thus, the second integral can be bounded by
\begin{align*}
    I_{2} &\le -\iint_{\Omega_{T}} \partial_t [\psi(u)]_{h} \,\dx\dt
    = \int_{\Omega \times \{0\}} \psi \big( u_o^{(\epsilon)} \big) \dx - \int_{\Omega \times \{T\}} [\psi(u)]_{h} \,\dx.
\end{align*}
Inserting the preceding estimates into \eqref{eq:split_time_derivative_integral}, and observing that the boundary terms with the integrand $\psi \big( u_o^{(\epsilon)} \big)$ cancel each other out, we get that
\begin{align*}
    \iint_{E} &\partial_t v_{h} \cdot \big( \power{v_{h}}{q} - \power{u}{q} \big) \,\dx\dt \\
    &\le \int_{\Omega \times \{ T \} } \psi([u]_{h} + s[\phi]_{h}) - [\psi(u)]_{h} \,\dx
    - s\iint_{E} \partial_t [\phi]_{h} \cdot \power{u}{q} \,\dx\dt.
\end{align*}
Since the first term on the right-hand side of the preceding inequality vanishes in the limit $h\downarrow0$ because of $u\in C^0([0,T];L^{q+1}(\Omega,\R^N))$ and $\phi=0$ on $\Omega\times\{T\}$,
we conclude that
\begin{align*}
    \limsup_{h\downarrow 0} \iint_{E} \partial_t v_{h} \cdot (\power{v_{h}}{q} - \power{u}{q}) \,\dx\dt \le -s \iint_{E} \partial_{t} \phi \cdot \power{u}{q}\,\dx\dt.
\end{align*}
Inserting this into \eqref{time_derivative_variational_inequality}, passing to the limit $h \downarrow 0$, and using that $f$ is convex gives us that
\begin{align*}
    \iint_{E} &f(x,u,Du) \,\dx\dt \\
	&\leq \iint_{E} f(x,u+s\phi,Du+sD\phi) \,\dx\dt -s \iint_{E} \partial_{t} \phi \cdot \power{u}{q}\,\dx\dt
	+\int_{E^0} \b \big[ u_o,u_o^{(\epsilon)} \big] \,\dx\\
    &\leq \iint_{E} (1-s)f(x,u,Du) +sf(x,u+\phi, Du + D\phi)\,\dx\dt    
    -s \iint_{E} \partial_{t} \phi \cdot \power{u}{q} \,\dx\dt
    \\ &\phantom{=}
	+\int_{E^0} \b \big[ u_o,u_o^{(\epsilon)} \big] \,\dx.
\end{align*}
Since $u_o^{(\varepsilon)} \to u_o$ in $L^{q+1}(E^0)$ as $\varepsilon \downarrow 0$, the last term on the right-hand side of the preceding inequality vanishes in the limit $\varepsilon \downarrow 0$.
Thus, reabsorbing the first term of the right-hand side of the preceding inequality into the left-hand side, and dividing by $s>0$ yields
\begin{align*}
    \iint_{E} \partial_{t} \phi \cdot \power{u}{q}+f(x,u,Du)  \,\dx\dt \le \iint_{E}  f(x,u+\phi, Du+D\phi)  \,\dx\dt.
\end{align*}
This concludes the proof of the lemma.
\end{proof}

\subsection{Proof of Theorem~\ref{thm:time_derivative_nondecreasing}}
Applying Lemma~\ref{lem:minimality_condition_for_var_sol} with $\phi$ replaced by $s\phi$ with some $s \in (0,1)$ to the variational solution $u$, and using \eqref{ineq:Lipschitz_condition} we obtain that
\begin{align*}
     \iint_{E}&  \power{u}{q} \cdot \partial_{t} \phi \,\dx\dt  \\
    &\le
    \iint_{E} \tfrac{1}{s} \big[ f(x,u+s\phi, Du+sD\phi) - f(x,u,Du) \big] \,\dx\dt  \\
    &\le
    c \iint_{E} \big( ( \vert Du \vert + \vert Du + s\phi \vert + \vert u \vert + \vert u+s\phi \vert )^{p-1} + \vert G \vert^{\frac{p-1}{p}} \big) (\vert D\phi\vert +\vert \phi \vert) \,\dx\dt \\
    &\to
    c \iint_{E} \big( ( \vert Du \vert + \vert u \vert )^{p-1} + \vert G \vert^{\frac{p-1}{p}} \big) (\vert D\phi\vert +\vert \phi \vert) \,\dx\dt
\end{align*}
in the limit $s\downarrow0$, 
for any $\phi \in C_{0}^{\infty}(E, \mathds{R}^{N})$. Since we can replace $\phi$ by $-\phi$ above, we obtain the same estimate for the absolute value of the left-hand side integral.  
Therefore, applying Hölder's inequality leads to 
\begin{align*}
    \bigg\vert \iint_{E}  \power{u}{q} \cdot \partial_{t} \phi \,\dx\dt \bigg\vert
    &\le
    c \iint_{E} \big( ( \vert Du \vert + \vert u \vert )^{p-1} + \vert G \vert^{\frac{p-1}{p}} \big) (\vert D\phi\vert +\vert \phi \vert) \,\dx\dt
    \\
    &\le c \bigg( \iint_{E} \vert Du \vert^{p} + \vert u \vert^{p} + \vert G \vert  \,\dx\dt \bigg)^\frac{p-1}{p} \Vert \phi \Vert_{V^{p}(E)} \\
    &\le
    c \Big[ \Vert u \Vert_{V^{p}(E)}^{p} + \Vert G \Vert_{L^1(E)} \Big]^{\frac{p-1}{p}} \Vert \phi \Vert_{V^{p}(E)} 
\end{align*}
for all $\phi \in C_{0}^{\infty}(E, \mathds{R}^{N})$.
Since \eqref{nondecreasing_condition} and \eqref{ineq:lower_measure_bound} hold, by Remark~\ref{rem:C_0_inf_is_dense_in_V_p,0} we know that $C_{0}^{\infty}(E, \mathds{R}^{N})$ is dense in $V^{p,0}(E)$.
Therefore, we obtain that $\partial_{t} \power{u}{q} \in \big( V^{p,0}(E) \big)^{\prime}$, as well as the claimed estimate in the $\big( V^{p,0}(E) \big)^{\prime}$-norm.

\bigskip
\textbf{Acknowledgements.}
This research was funded in whole or in part by the Austrian Science
Fund (FWF) 10.55776/J4853 and 10.55776/P36295. Jarkko Siltakoski was
supported by the Magnus Ehrnrooth and Emil Aaltonen foundations.
The second and third authors are grateful to the Department of
Mathematics at the Paris Lodron Universität Salzburg, where this
research was initiated, for their
hospitality during their stay.
For the purpose of open access, the authors have applied a CC BY public copyright license to any Author Accepted Manuscript (AAM) version arising from this submission.

\textbf{Data availability.}
Data sharing not applicable to this article as no datasets were generated or analyzed during the current study.

\textbf{Conflict of Interests.} The authors do not have a conflict of interests.

\end{document}